\documentclass[10pt,reqno]{amsart}
\usepackage{latexsym,amsmath}
\usepackage{times}
\usepackage{amsfonts}
\usepackage{amssymb}
\usepackage{color}
\usepackage[normalem]{ulem}

\usepackage{bbm}

\newenvironment{propositionproof}[1]{\noindent{\bf Proof of Proposition#1:\@}}{\hfill $\square$\\}
\newenvironment{claimproof}[1]{\noindent{\bf Proof of Claim #1:\@}}{\hfill $\square$\\}

\newcommand\nix{\,\cdot\,}
\newcommand\normA{\norm{\rho_{\nix\star}}_2}
\newcommand\normB{\norm{\rho_{\star\nix}}_2}
\newcommand\normC{\norm{\rho}_2}

\newcommand\Forb{\cF}

\newcommand{\beq}{\begin{equation}} \newcommand{\eeq}{\end{equation}}

\newcommand\gnp{G(n,p)}

\newcommand\gnm{\cG(n,m)}

\numberwithin{equation}{section}

\newcommand\br[1]{\left(#1\right)}

\def\vec#1{\mathchoice{\mbox{\boldmath$\displaystyle#1$}}
{\mbox{\boldmath$\textstyle#1$}}
{\mbox{\boldmath$\scriptstyle#1$}}
{\mbox{\boldmath$\scriptscriptstyle#1$}}}

\newcommand{\Ztame}{\widetilde Z^s_{k,\omega, \nu}}

\newcommand{\Zrho}{Z_{k,\rho}}
\newcommand{\Zw}{Z_{k,\omega}}
\newcommand{\Zwni}{Z_{k,\omega, \nu}^s}

\newcommand{\Zwnie}{Z_{k,\omega, \nu, \eta}^{s\,(2)}}
\newcommand{\Zwnis}{Z_{k,\omega, \nu, n^{-3/8}}^{s\,(2)}}

\newcommand{\Aw}{\cA_{k,\w}(n)}
\newcommand{\Awni}{\cA_{k,\w,\nu}^s(n)}
\newcommand{\Bw}{\cB_{k,\w}(n)}
\newcommand{\Bwni}{\cB_{k,\w,\nu}^s(n)}
\newcommand{\Bwnie}{\cB_{k,\w, \nu, \eta}^s(n)}
\newcommand{\Bwnis}{\cB_{k,\w, \nu, n^{-3/8}}^s(n)}
\newcommand{\Bwnisr}{\cB_{k,\w, \nu, n^{-3/8}}^s(n, \rho^0, \rho^1)}

\newcommand{\rwnik}{\rho^{k,\w,\nu,s}}
\newcommand{\Swn}{S_{k,\w,\nu}}

\newcommand{\dc}{d_{k,\mathrm{cond}}}

\newcommand{\dk}{d_{k,\mathrm{col}}}

\newcommand\COMB{Combinatorica}

\DeclareMathOperator{\pr}{\mathbb P}

\newtheorem{definition}{Definition}[section]
\newtheorem{claim}[definition]{Claim}

\newtheorem{remark}[definition]{Remark}
\newtheorem{theorem}[definition]{Theorem}
\newtheorem{lemma}[definition]{Lemma}
\newtheorem{proposition}[definition]{Proposition}
\newtheorem{corollary}[definition]{Corollary}

\newtheorem{fact}[definition]{Fact}

\usepackage{fullpage}

\newcommand\cA{\mathcal{A}}
\newcommand\cB{\mathcal{B}}
\newcommand\cC{\mathcal{C}}

\newcommand\cF{\mathcal{F}}
\newcommand\cG{\mathcal{G}}
\newcommand\cE{\mathcal{E}}

\newcommand\cH{\mathcal{H}}
\newcommand\cS{\mathcal{S}}

\newcommand\cV{\mathcal{V}}

\newcommand\cZ{\mathcal{Z}}

\def\cC{{\mathcal C}}
\def\cE{{\mathcal E}}

\newcommand\eps{\varepsilon}

\newcommand\Var{\mathrm{Var}}
\newcommand\Erw{\mathbb{E}}
\newcommand\w{{\omega}}

\newcommand{\Po}{{\rm Po}}

\newcommand{\bink}[2] {{{#1}\choose {#2}}}

\newcommand\ra{\rightarrow}

\newcommand\bc[1]{\left({#1}\right)}
\newcommand\cbc[1]{\left\{{#1}\right\}}
\newcommand\bcfr[2]{\bc{\frac{#1}{#2}}}
\newcommand{\bck}[1]{\left\langle{#1}\right\rangle}
\newcommand\brk[1]{\left\lbrack{#1}\right\rbrack}
\newcommand\scal[2]{\bck{{#1},{#2}}}
\newcommand\norm[1]{\left\|{#1}\right\|}

\newcommand{\whp}{w.h.p.}

\newcommand{\Erdos}{Erd\H{o}s}
\newcommand{\Renyi}{R\'enyi}

\newcommand\Lem{Lemma}
\newcommand\Prop{Proposition}
\newcommand\Thm{Theorem}
\newcommand\Cor{Corollary}
\newcommand\Sec{Section}

\begin{document}

\title{On the number of solutions in random graph $k$-colouring}

\author[Felicia Rassmann]{Felicia Rassmann$^*$}
\thanks{$^*$The research leading to these results has received funding from the European Research Council under the European Union's Seventh Framework Programme (FP/2007-2013) / ERC Grant Agreement n.\ 278857--PTCC}
\date{\today}

\address{Felicia Rassmann, {\tt rassmann@math.uni-frankfurt.de}, Goethe University, Institute of Mathematics, 10 Robert-Mayer-Str, Frankfurt 60325, Germany.}

\begin{abstract}
\noindent
Let $k \ge 3$ be a fixed integer. We exactly determine the asymptotic distribution of $\ln Z_k(G(n,m))$, where $Z_k(G(n,m))$ is the number of $k$-colourings of the random graph $G(n,m)$. A crucial observation to this aim is that the fluctuations in the number of colourings can be attributed to the fluctuations in the number of small cycles in $G(n,m)$. Our result holds for a wide range of average degrees, and for $k$ exceeding a certain constant $k_0$ it covers all average degrees up to the so-called {\em condensation phase transition}.

\noindent
%\emph{Key words:}	random structures, phase transitions, graph coloring.
%\emph{Mathematics Subject Classification:} 05C80 (primary), 05C15 (secondary)
\end{abstract}

\maketitle
\section{Introduction}

\subsection{Background and motivation}
Going back to the ground-breaking paper of \Erdos\ and \Renyi~\cite{ER} in 1960, the study of the random graph colouring problem has attained a lot of attention and innumerable articles have been published in this area of research over the years. In the most frequently studied model, a random graph $G(n,m)$ on the vertex set $\brk n=\cbc{1,\ldots,n}$ with precisely $m$ edges is drawn uniformly at random from all such graphs. 

A question that has turned out to be a very challenging one is how to choose $n$ and $m$ to obtain a random graph that is colourable \whp. Or, put differently, whether a random graph with given $n$ and $m$ can be coloured with a fixed number of colours, thus determining its chromatic number.
% A further natural question concerns the (typical) number of $k$-colourings of a random graph . 

Beginning in the 1990s, considerable progress has been made in the case of {\em sparse} random graphs, where $m=O(n)$ as $n\ra\infty$. Much effort has been devoted to studying the typical value of the chromatic number of $G(n,m)$~\cite{AchNaor,BBColor,LuczakColor,Matula} and its concentration~\cite{AlonKriv,Luczak,ShamirSpencer}. Several experiments and simulations led to the hypothesis, that, when changing the ratio of edges to variables, there is a transition from a regime where the random graph is colourable \whp~to the one where it is not \whp. Furthermore, the observation was that this transition does not happen smoothly, suggesting the existence of a sharp satisfiability threshold. Indeed, in 1999, Achlioptas and Friedgut~\cite{AchFried} proved the existence of a {\em sharp threshold sequence} $\dk(n)$ for any $k\geq3$, meaning that for any fixed $\eps>0$ the random graph $G(n,m)$ is $k$-colourable \whp~if $2m/n<\dk(n)-\eps$, whereas $G(n,m)$ fails to be $k$-colourable \whp~if $2m/n>\dk(n)+\eps$. This threshold sequence is non-uniform, i.e.~it is a function of $n$ and although it is broadly believed to converge for $n$ tending to infinity, this has not been established up to now. 
Also, in spite of continued efforts, the exact value of this threshold remains unknown up to date. The best current bounds~\cite{Covers,Danny} on $\dk(n)$ show that there is a sequence $(\gamma_k)_{k\geq3}$, $\lim_{k\ra\infty}\gamma_k=0$,
such that
\begin{equation*} %\label{eq_kcol}
(2k-1)\ln k-2\ln2-\gamma_k\leq\liminf_{n\ra\infty}\dk(n)\leq\limsup_{n\ra\infty}\dk(n)\leq(2k-1)\ln k-1+\gamma_k.
\end{equation*}  

Yet, there exist predictions by statistical physicists regarding the precise location of this threshold. They developed a method called {\em cavity method} that allowed them to gain insights into the combinatorial structure of the random graph colouring problem and to understand the significance of {\em typical} $k$-colourings, i.e.~$k$-colourings chosen uniformly at random from the set of all $k$-colourings, on both the combinatorial and algorithmic aspects of the problem \cite{pnas}. What is more, this method has also been used to predict a further phase transition shortly before the colouring threshold. This transition $\dc$ has been named {\em condensation} and its existence and location have rigorously been determined in 2014 by Bapst et al.~\cite{Cond}.
Under the assumption that $k\geq k_0$ for a certain constant $k_0$ it is possible to calculate the number $\dc$ precisely~\cite{Cond}, and an asymptotic expansion in $k$ yields
\begin{align*}\label{eq_dc}
\dc&=(2k-1)\ln k-2\ln2+\gamma_k,\qquad\mbox{ where }\lim_{k\ra\infty}\gamma_k=0.
\end{align*}

The condensation transition plays a very important role for several reasons. It marks the point where the behaviour of the number of solutions changes significantly, as does the geometry of the solution space~\cite{Barriers,Molloy}. The prediction states that while two $k$-colourings chosen uniformly at random tend to be uncorrelated before the condensation threshold, they typically exhibit long-range correlations afterwards~\cite{Reconstr}. Furthermore, the condensation transition persists, in contrast to the colourability transition, also for finite inverse temperatures \cite{BCOR}.
In recent work, it has been proved that the condensation transition is also related to the information theoretic threshold in the stochastic block model \cite{BanksMoore, NeemanNetra}, where it marks the point from which on it is possible to decide whether a random graph has been drawn from a planted distribution or not.

By obtaining an exact expression for the asymptotic distribution of the logarithm of the number of solutions up to the condensation threshold $\dc$, in the present paper we give a definite and complete answer to the question about the relationship between the planted model and the Gibbs distribution. Furthermore, we show that the fluctuations in the number of solutions can completely be attributed to the presence of short cycles, thereby eliminating the possibility of other influencing factors.  

For a graph $G$ on $n$ vertices, we let $Z_k(G)$ be the number of {\em k-colourings} (also called {\em solutions}) of $G$, which are maps $\sigma: \brk{n} \to \brk{k}$ such that $\sigma(i) \ne \sigma(j)$ for all edges $\{i,j\}$ of $G$. 
We always consider {\em sparse} random graphs $G(n,m)$ where $m=O(n)$. As we are going to need a very precise computation of the first and second moment of the number of $k$-colourings of $G(n,m)$, we distinguish the parameter $d^{\prime}$, which is such that $m = \lceil d^{\prime}n/2 \rceil$,
from $d=2m/n$, which arises naturally in the computations of the first and second moment. We note that $d^{\prime} \sim d$, although $d = d(n)$ might vary with $n$, whereas $d^{\prime}$ is assumed to be fixed as $n\to\infty$.

\subsection{Results}

We show that under certain conditions the number $Z_k(G(n,m))$ of $k$-colourings of the random graph is concentrated tightly and determine the distribution of $\ln Z_k(G(n,m))-\ln \Erw{\brk{Z_k(G(n,m))}}$ asymptotically in a density regime up to the condensation transition.  

Before we state the result, we introduce the following notation. For $k\geq3$, we define
	\begin{equation} \label{eq_dc}
	\dc=\sup\cbc{d^{\prime}>0:\liminf_{n\ra\infty}\Erw\brk{Z_k(G(n,m))^{1/n}}=k(1-1/k)^{d^{\prime}/2}}.
	\end{equation}
This definition is motivated by the well-known fact that
	\begin{align*}
	\Erw\brk{Z_k(G(n,m))}=\Theta\br{k^n(1-1/k)^{m}}.
	\end{align*}
Jensen's inequality shows that $\limsup_{n\ra\infty}\Erw\brk{Z_k(G(n,m))^{1/n}}\leq k(1-1/k)^{d^{\prime}/2}$ for all $d^{\prime}$ and this upper bound is tight up to the density $\dc$.

\begin{theorem}\label{Thm_main}
There is a constant $k_0>3$ such that the following is true.
Assume either  that $k\geq3$ and $d^{\prime}\leq 2(k-1)\ln(k-1)$ or that $k\geq k_0$ and $d^{\prime}<\dc$. Further, let 
\begin{align*}
\lambda_l=\frac{d^l}{2l}\quad\mbox{ and }\quad\delta_l=\frac{(-1)^l}{(k-1)^{l-1}}
\end{align*}
for $l\ge 2$. Let $(X_l)_l$ be a family of independent Poisson variables with $\Erw[X_l]=\lambda_l$, all defined on the same probability space. Then the random variable
$$W=\sum_{l\ge 3}X_l\ln(1+\delta_l)-\lambda_l\delta_l-d^2/(4(k-1))$$
satisfies $\Erw |W|<\infty$ and $\ln Z_k(G(n,m))-\ln \Erw{\brk{Z_k(G(n,m))}}$ converges in distribution to $W$.
\end{theorem}

\begin{remark}
	By definition, $W$ has an infinitely divisible distribution.
	It was shown in \cite{Janson} that the random variable $W^{\prime}=\exp\brk{W}$ converges almost surely and in $L^2$ with $\Erw\brk{W^{\prime}}=1$ and $\Erw\brk{{W^{\prime}}^2}=\exp\brk{\sum_l\lambda_l\delta_l^2}$. Thus, by Jensen's inequality it follows that $\Erw\brk{W}\le 0 $ and $\Erw\brk{W^2}\le \sum_l\lambda_l\delta_l^2$. 
\end{remark}

\subsection{Discussion and further related work.}\label{Sec_discussion_Paper4}

The crucial observation that the proof of \Thm~\ref{Thm_main} builds upon is that the fluctuations of $\ln Z_k(G(n,m))$
can be attributed to variations in the number of cycles of bounded length in the random graph.
This has first been observed in \cite{aco_plantsil} and has been used to determine the order of magnitude of the fluctuations of $\ln Z_k(G(n,m))$ in the random graph colouring problem. Following this result, the asymptotic distribution of the logarithm of the number of solutions has been established for random regular $k$-SAT \cite{aco_Wormald} and random hypergraph 2-colouring \cite{Rassmann}. 
Our result \Thm~\ref{Thm_main} refines the analysis from \cite{aco_Wormald} to obtain this asymptotic distribution also in random graph $k$-colouring in a broad density regime up to the condensation transition $\dc$ (for large values of $k$).
The proof combines the second moment arguments from Achlioptas and Naor~\cite{AchNaor} and its enhancements from~\cite{Cond,Danny} with the ``small subgraph conditioning''. This method was originally developed in~\cite{RobWor, RobinsonWormald} and extended by Janson \cite{Janson} to obtain limiting distributions. It has been frequently used in random regular graph problems (see \cite{Wormald_Survey} for an enlightening survey), e.g.~in \cite{Kemkes} and \cite{RegCol} to upper-bound the chromatic number of the random $d$-regular graph, as the sharp threshold result~\cite{AchFried} does not apply for this problem. More recently, is has also been used to obtain results in the stochastic block model \cite{NeemanNetra} and to determine the satisfiability threshold for positive 1-in-$k$-SAT \cite{Moore}.
Unfortunately, Janson's result does not apply directly in our case and instead we have to perform a variance analysis along the lines of \cite{RobinsonWormald}, very analogue to \cite{aco_Wormald, Rassmann}. The reason for this is that in contrast to \cite{aco_plantsil}, where only bounds on the fluctuation of $\ln Z_k$ were proven, we aim at a statement about its asymptotic {\em distribution}. Thus, in the present paper it does not suffice to consider colourings with {\em balanced} colour classes (with a deviation of $o(n^{-1/2})$ from their typical value), but we have to get a handle on all colourings providing a positive contribution. To this aim, we collect together colourings exhibiting similar colour class sizes. This results in the need to not only consider one random variable, but a growing number of random variables simultaneously. 
We expect that it is possible to apply a combination of the second moment method and small subgraph conditioning to a variety of further random constraint problems, such as e.g.~random $k$-NAESAT, random $k$-XORSAT or random hypergraph $k$-colourability. However, for asymmetric problems like the well-known benchmark problem random $k$-SAT, we expect that the logarithm of the number of satisfying assignments exhibits stronger fluctuations and we doubt that a result similar to ours can be established.

\subsection{Preliminaries and notation}
We always assume that $n\geq n_0$ is large enough for our various estimates to hold and denote by $[n]$ the set $\{1,...,n\}$. 

We use the standard $O$-notation when referring to the limit $n\ra\infty$.
Thus, $f(n)=O(g(n))$ means that there exist $C>0$, $n_0>0$ such that for all $n>n_0$ we have $|f(n)|\leq C\cdot|g(n)|$.
In addition, we use the standard symbols $o(\cdot),\Omega(\cdot),\Theta(\cdot)$.
In particular, $o(1)$ stands for a term that tends to $0$ as $n\ra\infty$.
Furthermore, the notation $f(n)\sim g(n)$ means that $f(n)=g(n)(1+o(1))$ or equivalently $\lim_{n\to\infty} f(n)/g(n)=1$. Besides taking the limit $n\to\infty$, at some point we need to consider the limit $\nu\to\infty$ for some number $\nu \in \mathbb N$. Thus, we introduce $f(n,\nu)\sim_{\nu} g(n,\nu)$ meaning that $\lim_{\nu\to\infty}\lim_{n\to\infty} f(n,\nu)/g(n,\nu)=1$.

If $p=(p_1,\ldots,p_l)$ is a vector with entries $p_i\geq0$, then we let
	$$\cH(p)=-\sum_{i=1}^lp_i\ln p_i.$$
Here and throughout, we use the convention that $0\ln0=0$.
Hence, if $\sum_{i=1}^lp_i=1$, then $\cH(p)$ is the entropy of the probability distribution $p$.
Further, for a number $x$ and an integer $h>0$ we let $(x)_h=x(x-1)\cdots(x-h+1)$ denote the $h$th falling factorial of $x$.

For the sake of simplicity, we choose to prove \Thm~\ref{Thm_main} using the random graph model $\gnm$.
This is a random (multi-)graph on the vertex set $[n]$ obtained by choosing $m$ edges $\vec e_1,\ldots,\vec e_m$ of the complete graph on $n$ vertices
uniformly and independently at random (i.e., with replacement). In this model we may choose the same edge more than once. However, the following statement shows that for sparse random graphs the probability of this event is bounded away from 1:

\begin{fact}\label{Fact_doubleEdge}
Assume that $m=m(n)$ is a sequence such that $m=O(n)$ and let $\cA_n$ be the event that $\gnm$ has no multiple edges. Then there is a constant $c>0$ such that $\lim_{n\to\infty}\pr\brk{\cA_n}>c$. 
\end{fact}

\section{Outline of the proof}\label{Sec_outline_Paper4}

To determine bounds on $Z_k(\gnm)$, it will be necessary to control the size of the colour classes. To formalize this, we introduce the following notation. For a map $\sigma:\brk n\ra\brk k$, we define
\begin{align*}
\rho(\sigma)=(\rho_1(\sigma),\ldots,\rho_k(\sigma)),\quad\mbox{where }\rho_{i}(\sigma) = | \sigma^{-1}(i) |/n \quad \text{for } i = 1 \dots k.
\end{align*}
Thus, $\rho(\sigma)$ is a probability distribution on $\brk k$, to which we refer as the {\em colour density} of $\sigma$.

Let $\cA_k(n)$ signify the set of all possible colour densities $\rho(\sigma)$ for $\sigma:\brk n\ra\brk k$.
Further, let $\cA_k$ be the set of all probability distributions $\rho=(\rho_1,\ldots,\rho_k)$ on $\brk k$,
and let $\rho^\star=(1/k,\ldots,1/k)$ signify the barycentre of $\cA_k$.

In order to simplify the notation, for the rest of the paper we assume that $\w, \nu$ are odd natural numbers, formally we define $N=\{2i-1:i \in\mathbb N\}$ and let $\w,\nu \in N$.
We say that $\rho=(\rho_1,\ldots,\rho_k)\in\cA_k(n)$ is {\em $(\w,n)$-balanced} if
\begin{align*}
\rho_i\in \left[\frac 1k- \frac{\w}{\sqrt{n}}\ ,\ \frac 1k+ \frac{\w}{\sqrt{n}}\right) \quad\mbox{ for all $i\in\brk k$}
\end{align*}
and let $\Aw$ denote the set of all $(\omega,n)$-balanced $\rho\in\cA_k(n)$. As we will see, in order to prove statements about the number $Z_k$ of all solutions, it suffices to consider solutions $\sigma$ with $\rho(\sigma) \in \Aw$. We let $\Zw(G)$ signify the number of {\em $(\omega,n)$-balanced $k$-colourings} of a graph $G$ on $[n]$, i.e. $k$-colourings $\sigma$ such that $\rho(\sigma) \in \Aw$.

Since verifying the required properties to apply small subgraph conditioning directly for the random variable $\Zw$ is very intricate, we break $\Zw$ down into smaller contributions, for which we determine the first and second moment in the following sections.

To this aim, we decompose the set $\Aw$ into smaller sets. We define
\begin{align}\label{eq_Swn}
\Swn=\cbc{s\in \mathbb Z^k:\|s\|_1=2i, i \in\mathbb N, i \le \frac{\w\nu-1}2}.
\end{align}

$\Swn$ contains vectors that we use as centres of disjoint 'balls' to partition the set $\Aw$: For $s=(s_1,...,s_k) \in \Swn$, we let $\rwnik \in \mathbb R^k$ be the vector with components 
\begin{align}\label{rwnik}
\rwnik_i=\frac 1k+\frac{s_i}{\nu\sqrt n}.
\end{align}
Further, we let $\Awni$ be the set of all colour densities $\rho\in\Aw$ such that
\begin{align*}
\rho_i\in \left[\rwnik_i- \frac{1}{\nu\sqrt{n}}\ ,\ \rwnik_i+ \frac{1}{\nu\sqrt{n}}\right).
\end{align*}
For a graph $G$, we denote by $\Zwni(G)$ the number of 2-colourings $\sigma$ such that $\rho(\sigma)\in\Awni$. For each fixed $\nu$, we have $\Zw=\sum_{s\in \Swn} \Zwni$ and our strategy is to apply small subgraph conditioning to the random variables $\Zwni$ rather than directly to $Z_k$. But first, we will calculate the first moments of $Z_k$ and $\Zw$ in \Sec~\ref{Sec_first_moment_Paper4} to obtain the following.

\begin{proposition} \label{Prop_first_mom}
Fix an integer $k\geq 3$ and a number $d^{\prime} \in (0, \infty)$. Let $\w>0$. Then 
$$\Erw\brk{Z_k(\gnm)}=\Theta(k^n(1-1/k)^m)\quad\mbox{and}\quad
	\lim_{\w \to \infty} \liminf_{n \to \infty}\frac{\Erw\brk{\Zw(\gnm)}}{\Erw\brk{Z_k(\gnm)}}=1. $$
\end{proposition}

As discussed in \Sec~\ref{Sec_discussion_Paper4}, the key observation the proof is based on is that the fluctuations of $Z_k(\gnm)$ can be attributed to fluctuations in the number of cycles of a bounded length.
Hence, for an integer $l\geq 2$ we let $C_{l,n}$ denote the number of cycles of length exactly $l$ in $\gnm$.
Let
\begin{align}\label{eq_lambdelt}
	\lambda_l=\frac{d^l}{2l}\quad\mbox{ and }\quad\delta_l=\frac{(-1)^l}{(k-1)^{l-1}}.
	\end{align}
The following fact shows that $C_{2,n},\ldots$ are asymptotically independent Poisson variables (e.g.~\cite[\Thm~5.16]{Bollobas}):

\begin{fact}\label{Fact_cyc}
If $c_2,\ldots,c_L$ are non-negative integers, then
	$$\lim_{n\ra\infty}\pr\brk{\forall 2\leq l\leq L:C_{l,n}=c_l}=\prod_{l=2}^L\pr\brk{\Po(\lambda_l)=c_l}.$$
\end{fact}

In \Sec~\ref{Sec_Short_cyc} the impact of the cycle counts $C_{l,n}$ on the first moment of $\Zwni(\gnm)$ is investigated. As this was already done in \cite{aco_plantsil}, we carry it out in the present work only for the sake of completeness. The result is the following:

\begin{proposition}\label{Prop_CondRatio1stMom}
Assume that $k\geq3$ and $d^{\prime}\in(0,\infty)$. 
Then
\begin{align*} %\label{eq_lambdeltConv}
\sum_{l=2}^\infty\lambda_l\delta_l^2<\infty.
\end{align*}
Moreover, let $\w, \nu \in N$ and $c_2,\ldots, c_L$ be non-negative integers. Then
\begin{align}\label{eq_CondRatio1stMom}
\frac{\Erw\brk{\Zwni(\gnm)|\forall 2\leq l\leq L:C_{l,n}=c_l}}{\Erw\brk{\Zwni(\gnm)}}\sim \prod_{l=2}^L\brk{1+\delta_l}^{c_l}\exp\brk{-\delta_l\lambda_l}.
\end{align}
\end{proposition}

Additionally, to apply small subgraph conditioning, we have to determine the second moment of $\Zwni(\gnm)$ very precisely. This step constitutes the main technical work of this paper. We consider two regimes of $d^{\prime}$ and $k$ separately. In the simpler case, based on the second moment argument from~\cite{AchNaor}, we obtain the following result.

\begin{proposition}\label{Prop_ratio_sec_first}
Assume that $k \geq 3$ and $d^{\prime} < 2(k-1) \ln (k-1)$.
Then
	$$\frac{ \Erw \brk{\Zwni(\gnm)^2}}{\Erw \brk{ \Zwni(\gnm)}^2} \sim \exp \brk{\sum_{l \geq 2} \lambda_l \delta_l^2}.  $$
\end{proposition}

\noindent
The second regime of $d^{\prime}$ and  $k$ is that $k\geq k_0$ for a certain constant $k_0\geq3$ and $d^{\prime}<\dc$ (with $\dc$ the number defined in~(\ref{eq_dc})).
In this case, we replace $\Zwni$ by the slightly tweaked random variable $\Ztame$ used in the second moment arguments from~\cite{Cond,Danny}.

\begin{proposition}\label{Prop_first_mom_tame_bal}
There is a constant $k_0\geq3$ such that the following is true.
Assume that $k\geq k_0$ and $2(k-1)\ln(k-1)\leq d^{\prime}<\dc$. Then for each $\w, \nu \in  N$ and $s \in \Swn$ there exists an integer-valued random variable $0\leq\Ztame\leq\Zwni$ such that
\begin{align}\label{eq_first_moment_tame_bal}
\Erw \brk{\Ztame(\gnm)} &\sim \Erw\brk{\Zwni(\gnm)}	\qquad\mbox{and}\\
\frac{\Erw\brk{\Ztame(\gnm)^2}}{\Erw\brk{\Ztame(\gnm)}^2} &\leq (1+o(1))\exp \brk{\sum_{l \geq 2} \lambda_l \delta_l^2}.	\nonumber
\end{align}
\end{proposition}

\noindent
The proofs of \Prop s~\ref{Prop_ratio_sec_first} and~\ref{Prop_first_mom_tame_bal} 
appear at the end of \Sec~\ref{Sec_second_moment_Paper4}.
In order to apply small subgraph conditioning to the random variable $\Ztame$, we need to investigate the impact of $C_{l,n}$ on the first moment of $\Ztame$. Thus, we need a similar result as \Prop~\ref{Prop_CondRatio1stMom} for $\Ztame$.
Fortunately, instead of having to reiterate the proof of \Prop~\ref{Prop_CondRatio1stMom}, we obtain the following by combining \Prop~\ref{Prop_CondRatio1stMom} with~\eqref{eq_first_moment_tame_bal}:

\begin{corollary} \label{Cor_first_moment_tame_bal}
Let $c_2,\ldots,c_L$ be non-negative integers. With the assumptions and notation of \Prop~\ref{Prop_first_mom_tame_bal} we have
\begin{equation*} %\label{eq_Cor}
\frac{\Erw\brk{\Ztame(\gnm)|\forall 2\leq l\leq L:C_{l,n}=c_l}}{\Erw\brk{\Ztame(\gnm)}}\sim \prod_{l=2}^L\brk{1+\delta_l}^{c_l}\exp\brk{-\delta_l\lambda_l}.
\end{equation*}
\end{corollary}

As the proof is nearly identical to the one in \cite{aco_plantsil}, we defer it to Appendix A.

\noindent The aim is now to derive \Thm~\ref{Thm_main} from \Prop s~\ref{Prop_first_mom}-\ref{Prop_ratio_sec_first}. 
The key observation is that the variance of the random variables $\Zwni$ is affected by the presence of cycles of bounded length and that this is the only significant influence. As a consequence, conditioning on the small cycle counts up to some preselected length reduces the variance of $\Zwni$. What is maybe surprising is that conditioning on the number of enough small cycles reduces the variance to any desired fraction of $\Erw[\Zwni]^2$.

As done in \cite{aco_Wormald, Rassmann}, the arguments we use are similar to the small subgraph conditioning from \cite{Janson, RobinsonWormald}. But we do not refer to any technical statements from \cite{Janson, RobinsonWormald} directly because instead of working only with the random variable $Z_k$ we need to control all $\Zwni$ for fixed $\w,\nu \in N$ simultaneously. In fact, ultimately we have to take $\nu\to\infty$ and $\w\to\infty$ as well.
Our line of argument follows the path beaten in \cite{aco_Wormald, Rassmann} and the following three lemmas are nearly identical to the ones derived there.

For $L>2$, let $\cF_L=\cF_{L,n}(d,k)$ be the $\sigma$-algebra generated by the random variables $C_{l,n}$ with $2\le l\le L$.
The set of all graphs can be divided into groups according to the small cycle counts: For each $L \ge 2$, the decomposition of the variance of $\Zwni$ yields
$$\Var\brk{\Zwni(\gnm)}=\Var\brk{\Erw\brk{\Zwni(\gnm)|\cF_L}}+\Erw\brk{\Var\brk{\Zwni(\gnm)|\cF_L}},$$
meaning that the variance can be written as the variance of the group mean plus the expected value of the variance within a group.
The term $\Var\brk{\Erw\brk{\Zwni(\gnm)|\cF_L}}$ accounts for the amount of variance induced by the fluctuations of the number of cycles of length at most $L$. The strategy when using small subgraph conditioning is to bound the second summand, which is the expected conditional variance
$$\Erw\brk{\Var\brk{\Zwni(\gnm)|\cF_L}}=\Erw\brk{\Erw\brk{{\Zwni(\gnm)}^2|\cF_L}-\Erw\brk{\Zwni(\gnm)|\cF_L}^2}.$$
In the following lemma we show that in fact in the limit of large $L$ and $n$ this quantity is negligible. This implies that conditioned on the number of short cycles the variance vanishes and thus the limiting distribution of $\ln \Zwni$ is just the limit of $\ln\Erw\brk{\Zwni|\cF_L}$ as $n,L \to \infty$. This limit is determined by the joint distribution of the number of short cycles.
\begin{lemma}\label{Lem_concen_1}
Let $k \ge 3$ and $d^{\prime} \in \br{0,\infty}$. For any $\w,\nu \in N$ and $s\in \Swn$, we have
\begin{align*}
 \limsup_{L\to\infty} \limsup_{n\to\infty}\Erw\brk{\frac{\Erw\brk{{\Zwni(\gnm)}^2|\cF_L}-\Erw\brk{\Zwni(\gnm)|\cF_L}^2}{\Erw\brk{\Zwni(\gnm)}^2}}=0.
\end{align*}
\end{lemma}
\begin{proof}
Fix $\w, \nu\in N$ and set $Z_s=\Zwni(\gnm)$. Using Fact~\ref{Fact_cyc} and equation (\ref{eq_CondRatio1stMom}) from \Prop~\ref{Prop_CondRatio1stMom} we can choose for any $\eps>0$ a constant $B=B(\eps)$ and $L\ge L_0(\eps)$ large enough such that for each large enough $n\ge n_0(\eps, B, L)$ we have for any $s\in\Swn$:
 \begin{align}\label{eq_lem_concen_1_1}
\nonumber  \Erw\brk{\Erw\brk{Z_s|\cF_L}^2}&\ge \sum_{c_1,...,c_L\le B} \Erw\brk{Z_s|\forall 2 \le l\le L: C_{l,n}=c_l}^2 \pr\brk{\forall 2 \le l\le L: C_{l,n}=c_l}\\
\nonumber   &\ge \exp\brk{-\eps}\Erw\brk{Z_s}^2 \sum_{c_1,...,c_L\le B} \ \prod_{l=2}^L \brk{(1+\delta_l)^{c_l}\exp\brk{-\lambda_l\delta_l}}^2\pr\brk{\Po(\lambda_l)=c_l}\\
\nonumber   &= \exp\brk{-\eps}\Erw\brk{Z_s}^2 \sum_{c_1,...,c_L\le B} \ \prod_{l=2}^L \frac{\brk{(1+\delta_l)^2\lambda_l}^{c_l}}{c_l!\exp\brk{2\lambda_l\delta_l+\lambda_l}}\\
            &\ge \Erw\brk{Z_s}^2 \exp\brk{-2\eps+ \sum_{l=2}^L \delta_l^2\lambda_l}.
 \end{align}
The tower property for conditional expectations and the standard formula for the decomposition of the variance yields
\begin{align*}
 \Erw\brk{Z_s^2}&=\Erw\brk{\Erw\brk{Z_s^2|\cF_L}}=\Erw\brk{\Erw\brk{Z_s^2|\cF_L}-\Erw\brk{Z_s|\cF_L}^2}+ \Erw\brk{\Erw\brk{Z_s|\cF_L}^2}
\end{align*}
and thus, using (\ref{eq_lem_concen_1_1}) we have
\begin{align}\label{eq_lem_concen_1_2}
\frac{\Erw\brk{\Erw\brk{Z_s^2|\cF_L}-\Erw\brk{Z_s|\cF_L}^2}}{\Erw\brk{Z_s}^2 }\le \frac{ \Erw\brk{Z_s^2}}{\Erw\brk{Z_s}^2}- \exp\brk{-2\eps+ \sum_{l=2}^L \delta_l^2\lambda_l}.
\end{align}
Finally, the estimate $\exp[-x]\ge 1-x$ for $|x|<1/8$ combined with (\ref{eq_lem_concen_1_2}) and \Prop~\ref{Prop_ratio_sec_first} implies that for large enough $\nu,n,L$ and each $s\in\Swn$ we have
\begin{align*}
 \frac{\Erw\brk{\Erw\brk{Z_s^2|\cF_L}-\Erw\brk{Z_s|\cF_L}^2}}{\Erw\brk{Z_s}^2 }\le 2\eps \exp\brk{\sum_{l=2}^\infty \delta_l^2\lambda_l}.
\end{align*}
As this holds for any $\eps>0$ and by equation \eqref{eq_lambdelt} the expression $\exp\brk{\sum_{l=2}^\infty \delta_l^2\lambda_l}$ is bounded, the proof of the lemma is completed by first taking $n\to\infty$ and then $L\to\infty$.
\end{proof}

\begin{lemma}\label{Lem_concen_3}
	For any $\alpha>0$, we have
	\begin{align*}
	\limsup_{L\to\infty}\limsup_{n\to\infty} \pr\brk{|Z_k(\gnm)-\Erw\brk{Z_k(\gnm)|\cF_L}|>\alpha\Erw\brk{Z_k(\gnm)}}=0.
	\end{align*}
\end{lemma}

\begin{proof}
	To unclutter the notation, we set $Z_k=Z_k(\gnm)$ and $\Zw=\Zw(\gnm)$. First we observe that \Prop~\ref{Prop_first_mom} implies that for any $\alpha>0$ we can choose $\w \in N$ large enough such that
	\begin{align}\label{eq_lem_concen_3_1}
	\liminf_{n\to\infty} \Erw\brk{\Zw}>(1-\alpha^2)\Erw\brk{Z_k}.
	\end{align}
	We let $\nu \in N$. To prove the statement, we need to get a handle on the cases where the variables $\Zwni(\gnm)$ deviate strongly from their conditional expectation $\Erw\brk{\Zwni(\gnm)|\cF_L}$. We let $Z_s=\Zwni(\gnm)$ and define 
	\begin{align*}
	X_s=|Z_s-\Erw\brk{Z_s|\cF_L}|\cdot \mathbf{1}_{\{|Z_s-\Erw\brk{Z_s|\cF_L}|>\alpha \Erw\brk{Z_s}\}}
	\end{align*}
	and $X=\sum_{s \in \Swn} X_s$. Then these definitions directly yield 
	\begin{align}\label{eq_lem_concen_3_2}
	\pr\brk{X<\alpha  \Erw\brk{\Zw}}\le \pr\brk{\left|\Zw-\Erw\brk{\Zw|\cF_L}\right|<2\alpha \Erw\brk{\Zw}}.
	\end{align}
	By the definition of the $X_s$'s and Chebyshev's inequality it is true for every $s$ that
	\begin{align*}
	\Erw\brk{X_s|\cF_L}&\le \sum_{j\ge 0} 2^{j+1} \alpha \Erw\brk{Z_s}\pr\brk{\left|Z_s-\Erw\brk{Z_s|\cF_L}\right|>2^j\alpha \Erw\brk{Z_s}}\le \frac{4\Var\brk{Z_s|\cF_L}}{\alpha \Erw\brk{Z_s}}.
	\end{align*}
Hence, using that with \Prop~\ref{Prop_first_mom} there is a number $\beta=\beta(\alpha, \w)$ such that $\Erw\brk{Z_s}/\Erw\brk{Z_k}\le \beta/(|\Swn|)$ for all $s\in \Swn$ and $n$ large enough, we have
	\begin{align*}%\label{eq_lem_concen_3_3}
	\Erw\brk{X|\cF_L}\le \sum_{s \in \Swn}\frac{4\Var\brk{Z_s|\cF_L}}{\alpha \Erw\brk{Z_s}}\le  \frac{4\beta \Erw\brk{Z_k}}{\alpha|\Swn|} \sum_{s \in \Swn}\frac{\Var\brk{Z_s|\cF_L}}{\Erw\brk{Z_s}^2}.
	\end{align*}
Taking expectations, choosing $\eps=\eps(\alpha, \beta, \w)$ small enough and applying \Lem~\ref{Lem_concen_1}, we obtain
	\begin{align}\label{eq_lem_concen_3_4}
	\Erw\brk X =\Erw\brk{\Erw\brk{X|\cF_L}} \le \frac{4\beta \Erw\brk{Z_k}}{\alpha|\Swn|} \sum_{s \in \Swn}\frac{\Erw\brk{\Var\brk{Z_s|\cF_L}}}{\Erw\brk{Z_s}^2} \le \frac{4\beta\eps\Erw\brk{Z_k}}{\alpha} \le \alpha^2\Erw\brk{Z_k}.
	\end{align}
	Using (\ref{eq_lem_concen_3_2}), Markov's inequality, (\ref{eq_lem_concen_3_4}) and (\ref{eq_lem_concen_3_1}), it follows that
	\begin{align}\label{eq_lem_concen_3_5}
	\pr\brk{\left|\Zw-\Erw\brk{\Zw|\cF_L}\right|<2\alpha \Erw\brk{\Zw}}\ge 1-2\alpha.
	\end{align}
	
	\noindent Finally, the triangle inequality combined with Markov's inequality and equations (\ref{eq_lem_concen_3_1}) and (\ref{eq_lem_concen_3_5}) yields
	\begin{align*}
	\pr\brk{\left|Z_k-\Erw\brk{Z_k|\cF_L}\right|>\alpha\Erw\brk{Z_k}} & \le\pr\brk{\left|Z_k-\Zw\right|+\left|\Zw-\Erw\brk{\Zw|\cF_L}\right|+\left|\Erw\brk{\Zw|\cF_L}-\Erw\brk{Z_k|\cF_L}\right|>\alpha\Erw\brk{Z_k}}\\
	&\le 3\alpha + \alpha/3 + 3\alpha < 7\alpha,
	\end{align*}
	which proves the statement. 
\end{proof}

\begin{lemma}\label{Lem_concen_2}
	Let 
	\begin{align}\label{eq_lem_concen_2_0}
	U_L&=\sum_{l=2}^L C_{l,n}\ln(1+\delta_l)-\lambda_l\delta_l.
	\end{align}
	Then $\limsup_{L\to\infty} \limsup_{n\to\infty} \Erw\brk{|U_L|}<\infty$ and further for any $\eps>0$ we have
	
	\begin{align}\label{eq_lem_concen_2_1}
	\limsup_{L\to\infty}\limsup_{n\to\infty} \pr\brk{|\ln \Erw\brk{Z_k(\gnm)|\cF_L}-\ln \Erw\brk{Z_k(\gnm)}-U_L|>\eps}=0
	\end{align}
\end{lemma}
\begin{proof}
	In a first step we show that $\Erw\brk{|U_L|}$ is uniformly bounded. As $x-x^2\le\ln(1+x)\le x$ for $|x|\le 1/8$ we have for every $l\le L$: 
	\begin{align*}
	\Erw\brk{\left|C_{l,n}\ln(1+\delta_l)-\lambda_l\delta_l\right|}\le \delta_l\Erw\brk{\left|C_{l,n}-\lambda_l\right|}+\delta_l^2\Erw\brk{C_{l,n}}.
	\end{align*}
	Therefore, Fact~\ref{Fact_cyc} implies that
	\begin{align}\label{eq_lem_concen_2_2}
	\Erw\brk{\left|U_L\right|}\le\sum_{l=2}^L\delta_l\sqrt{\lambda_l}+\delta_l^2\lambda_l.
	\end{align}
	\Prop~\ref{Prop_CondRatio1stMom} ensures that $\sum_l \delta_l^2\lambda_l <\infty$. Furthermore, as $d^{\prime}\le (2k-1)\ln k$, we have 
	$$\sum_l \delta_l\sqrt{\lambda_l}\le \sum_l k^l 2^{-(k-1)l/2} <\infty$$
	and thus (\ref{eq_lem_concen_2_2}) shows that $\Erw\brk{|U_L|}$ is uniformly bounded. 
	
	To prove (\ref{eq_lem_concen_2_1}), for given $n$ and a constant $B>0$ we let $\cC_B$ be the event that $C_{l,n}<B$ for all $l\le L$. Referring to Fact~\ref{Fact_cyc}, we can find for each $L,\eps>0$ a $B>0$ such that
	\begin{align}\label{eq_lem_concen_2_3}
	\pr\brk{\cC_B}>1-\eps.
	\end{align}
	To simplify the notation we set $Z_k=Z_k(\gnm)$ and $\Zw=\Zw(\gnm)$. By \Prop~\ref{Prop_first_mom} we can choose for any $\alpha>0$ a $\w>0$ large enough such that $\Erw\brk{\Zw}>(1-\alpha)\Erw\brk{Z_k}$ for large enough $n$. Then \Prop s~\ref{Prop_first_mom} and \ref{Prop_CondRatio1stMom} combined with Fact~\ref{Fact_cyc} imply that for any $c_1,...,c_L\le B$ and small enough $\alpha=\alpha(\eps, L, B)$ we have for $n$ large enough:
	\begin{align}\label{eq_lem_concen_2_4}
	\nonumber \Erw\brk{Z_k|\forall 2\leq l\leq L:C_{l,n}=c_l}& \ge \Erw\brk{\Zw|\forall 2\leq l\leq L:C_{l,n}=c_l}\\
	& \ge \exp\brk{-\eps}\Erw\brk{Z_k} \prod_{l=2}^L (1+\delta_l)^{c_l}\exp\brk{-\delta_l\lambda_l}.
	\end{align} 
On the other hand, for $\alpha$ sufficiently small and large enough $n$ we have
	\begin{align}\label{eq_lem_concen_2_5}
	 \nonumber  \Erw\brk{Z_k|\forall 2\leq l\leq L:C_{l,n}=c_l}& = \Erw\brk{Z_k-\Zw|\forall 2\leq l\leq L:C_{l,n}=c_l} + \Erw\brk{\Zw|\forall 2\leq l\leq L:C_{l,n}=c_l}\\
	  \nonumber &\le \frac{2\alpha \Erw\brk{Z_k}}{\prod_{l=2}^L\pr\brk{\Po(\lambda_l)=c_l}}+\Erw\brk{\Zw|\forall 2\leq l\leq L:C_{l,n}=c_l}\\
		    & \le \exp\brk{\eps}\Erw\brk{Z_k} \prod_{l=2}^L (1+\delta_l)^{c_l}\exp\brk{-\delta_l\lambda_l}
	\end{align}
Thus, the proof of (\ref{eq_lem_concen_2_1}) is completed by combining (\ref{eq_lem_concen_2_3}), (\ref{eq_lem_concen_2_4}), (\ref{eq_lem_concen_2_5}) and taking logarithms. 

\end{proof}

\begin{proof}[Proof of \Thm~\ref{Thm_main}]
For $L\ge 2$, we define
\begin{align*}
 W_L=\sum_{l=2}^L X_l \ln(1+\delta_l)-\lambda_l\delta_l \quad \text{ and } \quad W'=\sum_{l\ge 2} X_l \ln(1+\delta_l)-\lambda_l\delta_l.
\end{align*}
Then Fact~\ref{Fact_cyc} implies that for each $L$ the random variables $U_L$ defined in (\ref{eq_lem_concen_2_0}) converge in distribution to $W_L$ as $n\to\infty$. Furthermore, because $\sum_{l}\delta_l \sqrt{\lambda_l}, \sum_{l}\delta_l^2\lambda_l<\infty$, the martingale convergence theorem implies that $W'$ is well-defined and that the $W_L$ converge to $W'$ almost surely as $L\to\infty$. Hence, from \Lem s~\ref{Lem_concen_2} and \ref{Lem_concen_3} it follows that $\ln Z_k(\gnm)-\ln\Erw\brk{Z_k(\gnm)}$ converges to $W'$ in distribution, meaning that for any $\eps>0$ we have
\begin{align}\label{eq_proof_main_dist}
\lim_{n\to\infty} \pr\brk{|\ln Z_k(\gnm)-\ln\Erw\brk{Z_k(\gnm)}-W'|>\eps}=0.
\end{align}
To derive \Thm~\ref{Thm_main} from~(\ref{eq_proof_main_dist}), we denote by $S$ the event that $\gnm$ consists of $m$ distinct edges, or, equivalently, that no cycles of length 2 exist in $\gnm$. Given that $S$ occurs, $\gnm$ is identical to $G(n,m)$ and $W'$ is identical to $W$.
Furthermore, Fact~\ref{Fact_doubleEdge} implies that $\pr\brk{S}=\Omega(1)$. Consequently, \eqref{eq_proof_main_dist} yields 
\begin{align}\label{eq_proof_main_1_dist}
\nonumber 0&=\lim_{n\to\infty} \pr\brk{|\ln Z_k(\gnm)-\ln\Erw\brk{Z_k(\gnm)}-W'|>\eps|S}\\
&=\lim_{n\to\infty} \pr\brk{|\ln Z_k(G(n,m))-\ln\Erw\brk{Z_k(\gnm)}-W|>\eps}.
\end{align}
As \Lem~\ref{Lem_first_mom_balanced} implies that
$\Erw\brk{Z_k(\gnm)},\Erw\brk{Z_k(G(n,m)}=\Theta\br{k^n\br{1-1/k}^m}$, we have $\Erw\brk{Z_k(\gnm)}=\Theta(\Erw\brk{Z_k(G(n,m)})$ and with (\ref{eq_proof_main_1_dist}) it follows that
$$\lim_{n\to\infty} \pr\brk{|\ln Z_k(G(n,m))-\ln\Erw\brk{Z_k(G(n,m)))}-W|>\eps}=0,$$
which proves \Thm~\ref{Thm_main}.
\end{proof}

\section{The first moment}\label{Sec_first_moment_Paper4}

\noindent
The aim in this section is to prove \Prop~\ref{Prop_first_mom}.
The calculations that have to be done follow the path beaten in~\cite{AchNaor,Danny,Kemkes, Rassmann} and are in fact very similar to \cite{aco_plantsil}. Thus, most of the proofs are deferred to the appendix. 
Furthermore, at the end of the section we state a result that we need for \Prop~\ref{Prop_ratio_sec_first}.

Let $\Zrho(G)$ be the number of $k$-colourings of the graph $G$ with colour density $\rho$. Let $\rho^\star$ be a $k$-dimensional vector with all entries set to $1/k$. We define
	\begin{equation*}%\label{fm_function}
	f_1:\rho\in \cA_k\mapsto \cH(\rho) + \frac{d}{2} \ln \br{1 - \sum_{i=1}^k \rho_i^2 }.
	\end{equation*}
In order to determine the expectation of $\Zrho$, we have to analyse the function $f_1(\rho)$. The following lemma was already obtained in \cite{aco_plantsil} and its proof can be found in \Sec~\ref{Sec_appendix_first}.

\begin{lemma} \label{Lem_first_mom_balanced} 
Let $k \geq 3$ and $d^{\prime} \in (0, \infty)$. Then there exist numbers $C_1=C_1(k,d), C_2=C_2(k,d) > 0$ such that for any $\rho\in\cA_k(n)$ we have
\begin{equation}\label{eq_first_mom_balanced_1}	
C_1 n^{\frac{1-k}{2}} \exp \brk{n f_1(\rho)} \leq \Erw\brk{\Zrho(\gnm)} \leq C_2   \exp \brk{ n f_1(\rho)}.
\end{equation}
Moreover, if $\| \rho-\rho^\star\|_2 = o(1)$ and $d=2m/n$, then 
\begin{align}\label{eq_first_mom_balanced_2}
\Erw \brk{\Zrho(\gnm)} \sim  \br{{2 \pi n}}^{\frac{1-k}{2}} k^{k/2} \exp \brk{d/2+n f_1(\rho)}.
\end{align}
\end{lemma}

We can now state the expectation of $Z_k$. The proof will be carried out in detail in \Sec~\ref{Sec_appendix_first}.

\begin{corollary} \label{Cor_first_mom_total}
For any $k \geq 3, d^{\prime} \in (0, \infty)$ and $d=2m/n$, we have
$$ \Erw\brk{Z_k(\gnm)}\sim \exp \brk{d/2+ n f_1\br{\rho^\star}} \br{1 + \frac{d}{k-1}}^{-\frac{k-1}{2}}.$$
\end{corollary}

\begin{proof}[Proof of \Prop~\ref{Prop_first_mom}]
The first assertion is immediate from \Cor~\ref{Cor_first_mom_total}.
Moreover, the second assertion follows from \Cor~\ref{Cor_first_mom_total} and the second part of \Lem~\ref{Lem_first_mom_balanced}.
\end{proof}

\noindent Finally, as our approach requires the analysis of the random variables $\Zwni(\gnm)$, we derive an expression for $\Erw\brk{\Zwni(\gnm)}$ that we will need to prove \Prop~\ref{Prop_ratio_sec_first}.

\begin{lemma}\label{Lem_fm_shift}
 Let $k \ge 3, \w,\nu \in N, d^{\prime} \in (0,\infty)$ and $d=2m/n$. For $s \in \Swn$ and $\rwnik$ as defined in (\ref{rwnik}), we have
\begin{align*}
 \Erw\brk{\Zwni(\gnm)}\sim_{\nu} |\Awni|\br{{2 \pi n}}^{\frac{1-k}{2}} k^{k/2} \exp \brk{d/2+n f_1(\rwnik)}.
\end{align*}
\end{lemma}

\begin{proof}
Using a Taylor expansion of $f_1(\rho)$ around $\rho=\rwnik$, we get
 \begin{align}\label{eq_fm_shift}
 f_1(\rho)=f_1(\rwnik)+\Theta\br{\frac \w{\sqrt n}}\|\rho-\rwnik\|_1+\Theta\br{\|\rho-\rwnik\|_2^2}.
 \end{align}
As $\|\rho-\rwnik\|_1=O\br{\frac{1}{\nu \sqrt n}}$ for $\rho\in\Awni$ and $\|\rho-\rwnik\|_2^2=O\br{\frac{1}{\nu^2 n}}$, we conclude that  $f_1(\rho)=f_1(\rwnik)+O\br{\frac \w{\nu n}}$ and as this is independent of $\rho$, the assertion follows by inserting (\ref{eq_fm_shift}) in (\ref{eq_first_mom_balanced_2}) and multiplying by $|\Awni|$.

\end{proof}

\section{The second moment}\label{Sec_second_moment_Paper4}

The aim of this section is to prove \Prop~\ref{Prop_ratio_sec_first}, which constitutes the main technical contribution of this work 
and \Prop~\ref{Prop_first_mom_tame_bal}, which is done in the last subsection and is based on and an enhancement of results derived in \cite{AchNaor}.
The crucial points in our analysis are that, similar to \cite{aco_plantsil, Rassmann}, we need an asymptotically tight expression for the second moment and instead of confining ourselves to the case of colourings whose colour densities are $(O(1),n)$-balanced, as done in most of prior work ~\cite{AchNaor,Cond,Danny, Kemkes}, we need to deal with $(\omega,n)$-balanced colour densities for a diverging function $\omega=\omega(n)\ra\infty$. However, our work has to extend the calculations from \cite{aco_plantsil} following the example of \cite{Rassmann}, because we aim for a statement about the whole distribution of $\ln Z_k(G(n,m))$. Our line of argument follows that of \cite{Rassmann}, where analogue statements are proven for the problem of hypergraph 2-colouring. 

\subsection{Classifying the overlap}
To standardise the notation, we define the \textit{overlap matrix} $\rho(\sigma, \tau)=(\rho_{ij}(\sigma, \tau))_{i,j\in[k]}$  for two colour assignments $\sigma, \tau : [n] \to [k]$ as the doubly stochastic $k\times k$-matrix with entries
	$$ \rho_{ij}(\sigma, \tau) = \frac{1}{n} \cdot | \sigma^{-1}(i) \cap \tau^{-1}(j) |. $$
We let $\cB_k(n)$ denote the set of all overlap matrices and $\cB_k$ denote the set of all probability measures $\rho=(\rho_{ij})_{i,j\in[k]}$ on $[k]\times[k]$. Moreover, we let $\bar\rho$ signify the $k\times k$-matrix with all entries equal to $k^{-2}$, the barycentre of $\cB_k$. For a $k\times k$-matrix $\rho=(\rho_{ij})$, we introduce the shorthands
	$$ \rho_{i \star} = \sum_{j=1}^k \rho_{ij}, \qquad\rho_{\nix\star}=(\rho_{i \star})_{i\in[k]},
	 \qquad \qquad \rho_{\star j} = \sum_{i=1}^k \rho_{ij},\qquad
	 	\rho_{\star\nix}=(\rho_{\star i  })_{i\in[k]}. $$
With the notation from \Sec~\ref{Sec_outline_Paper4}, we observe that for any $\sigma, \tau : [n] \to [k]$ we have
 $\rho_{\nix\star},\rho_{\star\nix}\in\cA_{k}(n)$. 
We introduce the set
\begin{align*}
	\Bw&=\cbc{\rho\in\cB_k(n):
		\rho_{i \star},  \rho_{\star i} \in \left[\frac 1k -\frac{w}{\sqrt n},\frac 1k+ \frac{w}{\sqrt n}\right)\mbox{ for all $i\in [k]$}},
\end{align*}
which corresponds to $\Aw$ insofar as for $\rho \in \Bw$ we have $\rho_{i \star}, \rho_{\star i} \in \Aw$ for all $i \in [k]$. We remember $\Swn$ from \eqref{eq_Swn}. Then for $s\in \Swn$ we define
\begin{align*}
	\Bwni&=\cbc{\rho\in\Bw:
		\rho_{i \star},  \rho_{\star i} \in \left[\rwnik_i-\frac{1}{\nu \sqrt n},\rwnik_i+ \frac{1}{\nu \sqrt n}\right)\mbox{ for all $i\in [k]$}}.
\end{align*}
Thus, for any fixed $\nu$, $\Bw$ is a disjoint union of all $\Bwni$ for $s \in \Swn$.
For a given graph $G$ on $[n]$, we let $\Zrho^{(2)}(G)$ be the number of pairs $(\sigma,\tau)$ of $k$-colourings  of $G$ whose overlap 
is $\rho$.
By the linearity of expectation,
\begin{align}\label{eq_smm_linear}
\Erw \brk{\Zwni(\gnm)^2}&=\sum_{\rho\in\Bwni}\Erw\brk{\Zrho^{(2)}(\gnm)}.
\end{align}

To proceed calculating this quantity, we first need the following elementary estimates whose proofs can be found in \Sec~\ref{Sec_appendix_second}.

\begin{fact}\label{Fact_varZ}
For any $k \geq 3$, $d^{\prime} \in (0, \infty)$ and $d=2m/n$, the following estimates are true.
\begin{enumerate}
\item Let  $\rho \in \cB_k(n)$. Then
	\begin{align} \label{eq_Zrho}
	\Erw \brk{\Zrho^{(2)}(\cG(n,m))}
	\sim \frac{\sqrt{2 \pi} n^{\frac{1-k^2}{2}}} {\prod_{i,j=1}^k 			\sqrt{2 \pi \rho_{ij}}} \exp\brk{d/2+ n \cH(\rho) +m \ln( 1 -\normA^2-\normB^2+\normC^2)}.
	\end{align} 
\item For any $\rho \in \cB_k(n)$ with $\|\rho-\bar \rho\|_2^2=o(1)$, we have
	\begin{align} \label{eq_Z_om} 
	\Erw\brk{\Zrho^{(2)}(\cG(n,m))}
	\sim k^{k^2} \br{2 \pi n}^{\frac{1-k^2}{2}}\exp\brk{d/2 + n \cH(\rho) +m \ln( 1 -\normA^2-\normB^2+\normC^2)}.
	\end{align}
\end{enumerate}
\end{fact}

To simplify the notation, we introduce the function $f_2: \cB_k \to \mathbb R$ defined as
\begin{align}\label{eq_def_f_2}
f_2(\rho)=\cH(\rho) +\frac d2 \ln( 1 -\normA^2-\normB^2+\normC^2).
\end{align}

A direct consequence of Fact \ref{Fact_varZ} that will be used in the sequel is that for every $\rho\in\cB_k(n)$ we have
\begin{align}\label{eq_erw_Z_rho}
\Erw\brk{\Zrho^{(2)}(\gnm)}=\exp\brk{nf_2(\rho) +O(\ln n)}.
\end{align}

\subsection{Dividing up the hypercube}

To proceed, we refine equation \eqref{eq_smm_linear}. For each $\w, \nu \in  N, s \in \Swn$ and $\eta>0$, we introduce
\begin{align*}
	\Bwnie&=\cbc{\rho\in\Bwni:\norm{\rho-\bar\rho}_2\leq\eta}.
\end{align*}

We are going to show that the r.h.s.\ of~(\ref{eq_smm_linear}) is dominated by the contributions with $\rho$ ``close to'' $\bar\rho$ in terms of the euclidean norm.
More precisely, for a graph $G$ let
	$$\Zwnie(G)=\sum_{\rho\in\Bwnie}\Zrho^{(2)}(G)\qquad
		\mbox{for any }\eta>0.$$
Then the second moment argument performed in~\cite{AchNaor} fairly directly yields the following statement showing that overlap matrices that are far apart from $\bar \rho$ do asymptotically not contribute to the second moment.

\begin{proposition}\label{Prop_AchNaor}
Assume that $k \geq 3$ and $d^{\prime} < 2(k-1) \ln (k-1)$. Further, let $\w,\nu \in N$. Then for any fixed $\eta >0$ and any $s\in \Swn$, it holds that $$\Erw\brk{\Zwni(\gnm)^2} \sim \Erw\brk{\Zwnie(\gnm)}.$$
\end{proposition}

To prove this proposition, we first define a function  
\begin{equation*} \label{eq_f}
\bar f_2:\rho\in\Bw\ra\mathbb R,\quad\rho\mapsto \cH(\rho) + \frac{d}{2} \ln\bc{ 1 - \frac{2}{k} + \norm\rho_2^2 }.
\end{equation*}
The following lemma shows how $f_2$ defined in \eqref{eq_def_f_2} relates to $\bar f_2$.

\begin{lemma}\label{Lemma_f2bar}
	For $\rho=(\rho_{ij}) \in \Bw$, we have
	\begin{align*} %\label{Lemma_approx_f2}
	\exp\brk{nf_2(\rho)} \sim \exp\brk{n\bar f_2(\rho)+O\br{\w^2}}.
	\end{align*}
\end{lemma}
\begin{proof}
We define the function 
	$$\zeta(\rho)=f_2(\rho)- \bar f_2(\rho)$$
and derive an upper bound on $\zeta(\rho)$. By definition, for each $\rho \in \Bw$ there exist $\alpha=(\alpha_i)_{i\in [k]}$ and $\beta=(\beta_j)_{j \in [k]}$ such that $\rho_{i\star}=\frac 1k+\alpha_i$ and $\rho_{\star j}=\frac 1k+\beta_j$ for all $i,j \in [k]$ with $|\alpha_i|, |\beta_j| \le \frac{\w}{\sqrt n}$. Thus, 
	\begin{align*}
	f_2(\rho)= \cH(\rho)+\frac d2\ln\br{1-\|\bar \rho_{\nix \star}+\alpha\|_2^2-\|\bar \rho_{\star \nix}+\beta\|_2^2+\|\rho\|_2^2}.
	\end{align*}
	As we are only interested in the difference between $f_2$ and $\bar f_2$, we can reparametrise $\zeta$ as
	\begin{align*}
	\zeta(\alpha, \beta)=\frac d2 \ln\br{\frac{1-\|\bar \rho_{\nix \star}+\alpha\|_2^2-\|\bar \rho_{\star \nix}+\beta\|_2^2+\|\rho\|_2^2}{1-\frac 2k+\|\rho\|_2^2}}.
	\end{align*} 
	Differentiating and simplifying the expression yields
	$\frac{\partial \zeta}{\partial \alpha_i}(\alpha, \beta), \frac{\partial \zeta}{\partial \beta_j}(\alpha, \beta)=O\br{\frac{\w}{\sqrt n}}$ for all $i,j \in [k]$. 
	According to the fundamental theorem of calculus, it follows that
	$$\max_{\rho\in\Bw}|\zeta(\rho)|=\int_{-\w/\sqrt n}^{\w/\sqrt n} O\br{\frac{\w}{\sqrt n}} d\alpha_1=O\br{\frac{\w^2}{n}},$$
	completing the proof.
\end{proof}
\begin{proof}[Proof of \Prop~\ref{Prop_AchNaor}]
Equation \eqref{eq_erw_Z_rho} combined with \Lem~\ref{Lemma_f2bar} reduces our task to studying the function $\bar f_2(\rho)$.
For the range of $d$ covered by \Prop~\ref{Prop_AchNaor}, this analysis is the main technical achievement of~\cite{AchNaor},
where (essentially) the following statement is proved.

\begin{lemma}\label{Lemma_AchNaor_Main}
Assume that $k \geq 3, \w\in N$ as well as $d^{\prime} \leq 2(k-1) \ln (k-1)$ and $d=2m/n$. For any $n >0$ and any overlap matrix $\rho \in \Bw$, we have
	\begin{align}\label{eqAchNaor_Main}
	\bar f_2(\rho) \leq \bar f_2(\bar{\rho}) - \frac{2(k-1) \ln (k-1) - d}{4(k-1)^2}\br{k^2 \| \rho \|_2^2 -1} + o(1). 
	\end{align}	
\end{lemma}
\begin{proof}
For $\rho$ such that $\sum_{i=1}^k \rho_{ij} = \sum_{i=1}^k \rho_{ji} = 1/k$,
the bound~(\ref{eqAchNaor_Main}) is proved in \cite[\Sec~3]{AchNaor}.
This implies that~(\ref{eqAchNaor_Main}) also holds for $\rho\in\Bw$, because $\bar f_2$ is uniformly continuous on the compact set $\Bw$.
\end{proof}

Now, assume that $k$ and $d$ satisfy the assumptions of \Prop~\ref{Prop_AchNaor} and let $\nu \in N$ and $\eta >0$ be any fixed number.
Then, for any $\hat\rho \in \Bwni$, we have $\| \hat \rho- \bar{\rho}\|_2 =O\br{\frac{\w}{\sqrt n}}$. Consequently, we obtain with \eqref{eq_erw_Z_rho} that
\begin{align}\label{eqAchNaor_12}
\sum_{ \substack{\rho \in \Bwni \\ \| \rho - \bar{\rho} \|_2 \leq \eta  }} \Erw \brk{\Zrho^{(2)} (\gnm)}\geq  
\Erw \brk{Z_{k,\hat \rho}^{(2)} (\gnm)}\geq 
\exp \brk{n f_2(\bar{\rho}) +O(\ln n)}.
\end{align}

On the other hand, the function $\cB \to \mathbb{R}$,  $\rho \to k^2 \| \rho\|_2$ is smooth, strictly convex and attains its global minimum of $1$ at $\rho = \bar{\rho}$.
Consequently, there exist $(c_k)_k >0$ such that if $\| \rho - \bar{\rho} \|_2 > \eta$, then $\br{k^2\|\rho\|_2-1} \geq c_k$.
Hence, Fact~\ref{Fact_varZ}, \Lem~\ref{Lemma_f2bar} and \Lem~\ref{Lemma_AchNaor_Main} yield
	\begin{align}\label{eqAchNaor_11}
	\sum_{ \substack{\rho \in \Bwni\\ \| \rho - \bar{\rho} \|_2 > \eta  }}
		\Erw \brk{\Zrho^{(2)}(\gnm)} \leq \exp \brk{n f_2(\bar{\rho}) - n c_k d_k + o(n)},\quad\mbox{where }
			d_k = \frac{2(k-1) \ln (k-1)- d}{4 (k-1)^2} > 0.
	\end{align}

Combining~(\ref{eqAchNaor_11}) and~(\ref{eqAchNaor_12}), we conclude that
$\Erw\brk{\Zwni(\gnm)^2} \sim \Erw\brk{\Zwnie(\gnm)}$, 
thereby completing the proof of \Prop~\ref{Prop_AchNaor}.

\end{proof}

Having reduced our task to studying overlaps $\rho$ such that $\norm{\rho-\bar\rho}_2\leq\eta$ for a small but fixed $\eta>0$,
in this section we are going to argue that, in fact, it suffices to consider $\rho$ such that
	$\norm{\rho-\bar\rho}_2\leq n^{-3/8}$
(where the constant $3/8$ is somewhat arbitrary; any number smaller than $1/2$ would do).
More precisely, we have

\begin{proposition} \label{Prop_second_mom_new_v}
Assume that $k \geq 3$ and that $d^{\prime} < \dc$. Let $\nu,\w \in N$ and $s\in \Swn$. There exists a number $\eta_0=\eta_0(d^{\prime},k)$
such that for any $0<\eta<\eta_0$ we have % 
$$ \Erw \brk{\Zwnie(\gnm)} \sim \Erw \brk{\Zwnis(\gnm)}. $$
\end{proposition}

The key to proving this proposition is the following lemma. It specifies the expected number of pairs of solutions in the cases where the overlap matrices $\rho \in \Bwni$ satisfy $\| \rho - \bar{\rho} \|_2 \leq n^{-3/8}$ or $\| \rho - \bar{\rho} \|_2 \in (n^{-3/8}, \eta)$.

\begin{lemma}\label{Lem_exp_saddle}
Let $k \geq 3, d^{\prime} < (k-1)^2 $ and $d=2m/n$. Set
\begin{align}\label{eq_CD}
C_n(d,k) = \exp\brk{d/2} k^{k^2} (2 \pi n)^{\frac{1-k^2}{2}} \quad  \text{ and } \quad D(d,k) = k^2 \br{1 - \frac{d}{(k-1)^2}}.
\end{align}
\begin{itemize}
\item If $\rho \in \Bwnie$ satisfies $\| \rho - \bar{\rho} \|_2 \leq n^{-3/8}$, then 
	\begin{equation} \label{eq_aux_bound_Z_rho_1}
		\Erw \brk{\Zrho^{(2)} (\gnm)} \sim C_n(d,k) \exp \brk{2 n f_1(\rho^\star) -n  \frac{D(d,k)}{2}  \| \rho - \bar{\rho}\|_2^2}.
	\end{equation}
\item There exist numbers $\eta=\eta(d,k)  >0$ and $A=A(d,k) >0$ such that if $ \rho \in \Bwnie$ satisfies $\| \rho - \bar{\rho} \|_2 \in (n^{-3/8}, \eta)$, then
	\begin{equation} \label{eq_aux_bound_Z_rho_2}
		\Erw \brk{\Zrho^{(2)} (\gnm)} = \exp \brk{2 n f_1(\rho^\star) - A n^{1/4}}.
	\end{equation}
\end{itemize}
\end{lemma}
\begin{proof} 

As Fact~\ref{Fact_varZ} yields
	$ \Erw \brk{\Zrho^{(2)} (\cG(n,m))}\sim C_n(d,k) \exp \brk{n f_2(\rho)},$ 
we have to analyse $f_2$. Expanding this function around $\bar{\rho}$ yields
\begin{align}\label{eq_exp_f}
 f_2(\rho)
 & = f_2(\bar{\rho}) - \frac{D(d,k)}{2} \| \rho-\bar \rho \|_2^2 + O(\|\rho-\bar{\rho}\|_2^3).  
\end{align}
Consequently, for $\| \rho - \bar{\rho} \|_2 \leq n^{-3/8}$,
	$$\exp \brk{n  f_2(\rho)} = \exp \brk{n f_2(\bar{\rho}) - n \frac{D(d,k)}{2} \| \rho- \bar{\rho}\|_2^2  + O(n^{-1/8})}.$$
As $f_2$ satisfies $f_2(\bar{\rho}) = 2 f_1(\rho^\star)$, the statement in \eqref{eq_aux_bound_Z_rho_1} follows.

To prove \eqref{eq_aux_bound_Z_rho_2}, we observe that similarly to \eqref{eq_exp_f} and because $f_2$ is smooth in a neighbourhood of $\bar{\rho}$, 
there exist $\eta >0$ and $A >0$ such that for $\| \rho - \bar{\rho} \|_2 \leq \eta$, 
$$ f_2(\rho) \leq f_2(\bar{\rho}) - A\| \rho - \bar{\rho} \|_2^2 . $$
Hence, if $\| \rho - \bar{\rho} \|_2 \in (n^{-3/8}, \eta)$, then
		 $$\Erw \brk{\Zrho^{(2)}(\gnm)} = O \br{n^{\frac{1-k^2}{2}}}
		 	\exp \brk{n f_2(\rho)}\leq %
					\exp \brk{2 n f_1(\rho^\star) - A n^{1/4}},$$ 
as claimed.
 \end{proof}

\begin{proof}[Proof of \Prop~\ref{Prop_second_mom_new_v}] We fix $s \in \Swn$. Further, we fix $\eta >0$ and $A>0$ as given by \Lem~\ref{Lem_exp_saddle}. 
For each $\hat \rho \in \Bwnie$, we have $\| \hat\rho- \bar{\rho}\|_2 =O\br{\frac{\w}{\sqrt{n}}}$ and obtain from the first part of \Lem~\ref{Lem_exp_saddle} that
	\begin{align}\label{eqVV01}
	\Erw \brk{\Zwnis(\gnm)}  \geq \Erw \brk{Z^{(2)}_{k,\rho_0} (\gnm)}\sim C_n(d,k) \exp \brk{2 n f_1(\rho^\star)+O\br{\w^2} }.
	\end{align}
On the other hand, because $|\Bwnie|$ is bounded by a polynomial in $n$,
the second part of \Lem~\ref{Lem_exp_saddle} yields
\begin{align} \label{eqVV02}
\sum_{\substack{\rho \in \Bwnie \\  \| \rho - \bar{\rho} \|_2 > n^{-3/8} }} \Erw \brk{\Zrho^{(2)}(\gnm)} &  \leq  \exp \brk{2 n f_1(\rho^\star)-A n^{1/6} + O(\ln n)} . %
  \end{align} 
Combining~(\ref{eqVV01}) and~(\ref{eqVV02}), we obtain
 \begin{align*} \Erw \brk{\Zwnie(\gnm)}  &\sim  \sum_{\substack{\rho \in \Bwnis}} \Erw \brk{\Zrho^{(2)}(\gnm)}\sim 
 \Erw \brk{\Zwnis(\gnm)}, \end{align*}
as claimed.
 \end{proof}

\subsection{Calculating the constant}
This section is dedicated to computing the contribution of the overlap matrices $\rho \in \Bwnis$. To this aim, we first show that in each region of the hypercube we can approximate $f_2$ by a function where the marginals are set to those of the centre of this region as defined in \eqref{rwnik}. More formally, let $f_2^s: \cB_k \to \mathbb R$ be defined as
\begin{align*} %\label{f_2s}
f_2^s:\rho \mapsto \cH(\rho)+\frac d2 \ln\br{1-2\|\rwnik\|_2^2+\|\rho\|_2^2}.
\end{align*}
Then the following is true
\begin{lemma}\label{Lem_f_2s}
	Let $k \ge 3, \w,\nu \in N$ and $C_n(d,k)$ as in \eqref{eq_CD}. Then for $\rho \in \Bwnis$ it holds that
	\begin{align*}
	\Erw \brk{\Zrho^{(2)}(\gnm)} &\sim C_n(d,k)\exp\brk{nf_2^s(\rho)+O\br{\frac{\w}{\nu}}}.
	\end{align*}
\end{lemma}
\begin{proof}
Equation (\ref{eq_Z_om}) of Fact~\ref{Fact_varZ} yields that
	\begin{align}\label{eq_sm_shift}
	\Erw \brk{\Zrho^{(2)}(\gnm)} \sim C_n(d,k)\exp\brk{nf_2(\rho)}.
	\end{align}
\noindent For $s \in \Swn$, we define the function 
	$$\zeta^s(\rho)=f_2(\rho)- f_2^s(\rho).$$
	To derive an upper bound on $\zeta^s(\rho)$ for all values $\rho \in \Bwnis$, we first we observe that there exist $\alpha=(\alpha_i)_{i\in [k]}$ and $\beta=(\beta_j)_{j \in [k]}$ such that the function $f_2$ can be expressed by setting $\rho_{i\star}=\rwnik_i+\alpha_i$ and $\rho_{\star j}=\rwnik_j+\beta_j$ for all $i,j \in [k]$ with $|\alpha_i|, |\beta_j| \le \frac{1}{\nu\sqrt n}$. Thus, 
	\begin{align*}
	f_2: \rho \mapsto \cH(\rho)+\frac d2\ln\br{1-\|\rwnik+\alpha\|_2^2-\|\rwnik+\beta\|_2^2+\|\rho\|_2^2}.
	\end{align*}
	As we are only interested in the difference between $f_2$ and $f_2^s$, we can reparametrise $\zeta^s$ as
	\begin{align*}
	\zeta^s(\alpha, \beta)=\frac d2 \ln\br{\frac{1-\|\rwnik+\alpha\|_2^2-\|\rwnik+\beta\|_2^2+\|\rho\|_2^2}{1-2\|\rwnik\|_2^2+\|\rho\|_2^2}}.
	\end{align*} 
	Differentiating and simplifying the expression yields
	$\frac{\partial \zeta^s}{\partial \alpha_i}(\alpha, \beta), \frac{\partial \zeta^s}{\partial \beta_j}(\alpha, \beta)=O\br{\frac{\w}{\sqrt n}}$ for all $i,j \in [k]$. 
	According to the fundamental theorem of calculus it follows for every $s \in \Swn$ that
	$$\max_{\rho\in\Bwnis}|\zeta^s(\rho)|=\int_{-\br{\nu\sqrt n}^{-1}}^{\br{\nu \sqrt n}^{-1}} O\br{\frac{\w}{\sqrt n}} d\alpha_1=O\br{\frac{\w}{n\nu}}.$$
	
	\noindent Combining this with (\ref{eq_sm_shift}) yields the assertion.
\end{proof}

Now we are able to give a very precise expression for the second moment. 

\begin{proposition}\label{Prop_second_mom_balanced_exact} Assume that
 $k \geq 3, \w,\nu \in N, d^{\prime} < (k-1)^2$ and $d=2m/n$. Let $s \in \Swn$. Then
\begin{align*}
\nonumber &\Erw \brk{\Zwnis(\gnm)}\\
& \qquad \qquad \sim_\nu \br{|\Aw| \br{{2 \pi n}}^{\frac{1-k}{2}} k^{k/2} \exp \brk{n f_1(\rwnik)}}^2 \exp\brk{d/2} \br{1- \frac{d}{(k-1)^2}}^{-\frac{(k-1)^2}{2} }.  
\end{align*}
\end{proposition}
		
The rest of this subsection will be dedicated to proving this proposition. 
In due course we are going to need the set of matrices with coefficients in $\frac1n\mathbb{Z}$ whose lines and columns sum to zero:
\begin{equation} \label{eq_def_cEn}
\cE_n = \cbc{ \br{\epsilon_{i,j}}_{\substack{1 \leq i \leq k \\ 1 \leq j \leq k} },\;\; \forall i, j \in [k],\; \epsilon_{i,j} \in \frac1n\mathbb{Z},\;
	\; \forall j \in [k],\;  \sum_{i = 1}^k \epsilon_{i j} = \sum_{i=1}^k \epsilon_{j i} = 0  } .
\end{equation}

The following result regards Gaussian summations over matrices in $\cE_n$. 

\begin{lemma}\label{Lem_hessian_wormald}
Let $k \geq 2$, $d^{\prime} < (k-1)^2$ and $D >0$ be fixed. Then
	\begin{equation*}
	\sum_{\epsilon \in \cE_n} \exp \brk{- n \frac{D}{2} \| \epsilon\|_2^2 + o(n^{1/2}) \| \epsilon\|_2}\sim \br{\sqrt{ {2 \pi}{n} }}^{(k-1)^2} D^{- \frac{(k-1)^2}{2}} k^{-(k-1)}.
	\end{equation*}
\end{lemma}

\Lem~\ref{Lem_hessian_wormald} and its proof are very similar to an argument used in~\cite[\Sec~3]{Kemkes}. In fact, \Lem~\ref{Lem_hessian_wormald} follows from

\begin{lemma}[{\cite[\Lem~6 (b) and 7 (c)]{Kemkes}}]\label{Lemma_WormaldMatrix}
	There is a $(k-1)^2\times(k-1)^2$-matrix $\cH=(\cH_{(i,j),(k,l)})_{i,j,k,l\in[k-1]}$ such that for any $\eps=(\eps_{ij})_{i,j\in[k]}\in\cE_n$ we have
	$$\sum_{i,j,i',j'\in[k-1]}\cH_{(i,j),(i',j')}\eps_{ij}\eps_{i'j'}=\norm\eps_2^2.$$
	This matrix $\cH$ is positive definite and $\det\cH = k^{2(k-1)} $.
\end{lemma}

The Proof of \Lem~\ref{Lem_hessian_wormald} can be found in \Sec~\ref{Sec_appendix_second}. 
		
Now we are ready to prove \Prop~\ref{Prop_second_mom_balanced_exact}.	

\begin{proof}[Proof of \Prop~\ref{Prop_second_mom_balanced_exact}]
\Lem~\ref{Lem_f_2s} states that for every $\rho \in \Bwnis$ we have
 \begin{align}\label{eq_sm_Erw}
\Erw \brk{\Zrho^{(2)}(\gnm)} &\sim C_n(d,k)\exp\brk{nf_2^s(\rho)+O\br{\frac{\w}{\nu}}}.
 \end{align}
Thus, all we have to do is analysing the function $f_2^s$ for $s \in \Swn$. To this aim, we expand $f_2^s(\rho)$ around $\rho=\rho^s$ where $\rho^s=(\rho^s_{ij})_{i,j}$ with $\rho_{ij}=\rwnik_i\cdot \rwnik_j$. Then with $D(d,k)$ as defined in \eqref{eq_CD} we have
\begin{align}\label{eq_sm_expand}
f_2^s(\rho)=f_2^s\br{\rho^s}+\Theta\br{\frac \w n}\|\rho-\rho^s\|_2-\frac{D(d,k)}2\|\rho-\rho^s\|_2^2+o(n^{-1}).
\end{align}
Combining \eqref{eq_sm_expand} with \eqref{eq_sm_Erw}, we find that
\begin{align}\label{eq_sm_Zrho}
\Erw \brk{\Zrho^{(2)}(\gnm)} &\sim C_n(d,k)\exp\brk{nf_2^s\br{\rho^s}+\Theta\br{\w}\|\rho-\rho^s\|_2-n\frac{D(d,k)}2\|\rho-\rho^s\|_2^2+O\br{\frac{\w}{\nu}}}.
\end{align} 
For two vectors of ``marginals'' $\rho^0, \rho^1 \in \Bwni$, we introduce the set of overlap matrices 
$$\Bwnisr = \{ \rho \in \Bwnis: \rho_{\nix\star}= \rho^0, 
		\rho_{\star\nix} = \rho^1 \}.$$
and observe that with this definition we have
\begin{align} \label{eq_sm_Comp}
 \Erw \brk{\Zwnis(\cG(n,m))} = \sum_{ \rho^0, \rho^1 \in \Bwni} \sum_{\rho \in \Bwnisr} \Erw \brk{\Zrho^{(2)}(\gnm)}.
\end{align}
In particular, the set $\Bwnisr$ contains the ``product'' overlap $\rho^0 \otimes \rho^1$ defined by
$ ( \rho^0 \otimes \rho^1)_{ij} = \rho^0_i \rho^1_j $
for $i,j \in [k]$. 
To proceed, we fix two colour densities $\rho^0, \rho^1 \in \Bwni$ and simplify the notation by writing 
$$ \widehat{\cB} = \Bwnisr, \qquad \widehat{\rho} = \rho^0 \otimes \rho^1. $$
Thus, the inner sum from \eqref{eq_sm_Comp} simplifies to 
$$\cS_1=\sum_{\rho \in \widehat{\cB}} \Erw \brk{\Zrho^{(2)}(\gnm)}.$$
and we are going to evaluate this quantity. We observe that with $\cE_n$ as defined in \eqref{eq_def_cEn}, for each $\rho\in\widehat{\cB}$ we can find  $\eps\in \cE_n$ such that 
 \begin{align*}
 \rho=\widehat\rho+\eps.
 \end{align*}
Hence, this gives $\left\|\rho-\rho^s\right\|_2=\left\|\widehat\rho+\eps-\rho^s\right\|_2$ and the triangle inequality yields
 \begin{align*}
 \left\|\eps\right\|_2-\left\|\widehat\rho-\rho^s\right\|_2\le \left\|\widehat\rho+\eps-\rho^s\right\|_2 \le \left\|\eps\right\|_2+\left\|\widehat\rho-\rho^s\right\|_2.
 \end{align*}
 By definition of $\widehat \rho$ and $\rho^s$, we have $\left\|\widehat\rho-\rho^s\right\|_2\le\frac{1}{\nu\sqrt n}$ and consequently 
 \begin{align}\label{eq_sm_Norm}
 \left\|\rho-\rho^s\right\|_2=\left\|\boldsymbol \eps\right\|_2+O\br{\frac{1}{\nu\sqrt n}}.
 \end{align}
Observing that $f_2^s\br{\rho^s}=\br{f_1(\rwnik)}^2$  and inserting \eqref{eq_sm_Norm} into \eqref{eq_sm_Zrho} while taking first $n \to \infty$ and afterwards $\nu \to \infty$, we obtain
\begin{align} \label{eq_sm_S_1}  
\cS_1
&  \sim_{\nu} C_n(d,k) \exp \brk{2 n f_1^s\br{\rwnik}}\sum_{\substack{\rho \in \widehat{\cB}}}  \exp \brk{- n \frac{D(d,k)}{2}\left\| \eps\right\|_2^2+o(n^{1/2})\left\|\eps\right\|_2 }.
\end{align}
To apply \Lem~\ref{Lem_hessian_wormald}, we have to relate $\rho \in \widehat{\cB}$ to $\eps \in \cE_n$. From the definitions we obtain
$$\left \{ \widehat{\rho} + \eps:  \eps \in \cE_n, \| \eps \|_2  \leq {n^{-3/8}}/2  \right \} \subset 
	\left \{ \rho \in \widehat{\cB}\right \} 
	\subset \left \{ \widehat{\rho} + \eps:  \eps \in \cE_n \right  \}. $$
We show that the contribution of $\eps \in \cE_n$ with $\| \eps \|_2  > {n^{-3/8}}/2 $ is negligible:  
 \begin{align*}
 \cS_2&= C_n(d,k) \exp \brk{2 n f_1^s\br{\rwnik}} \sum_{\substack{ \epsilon \in \cS_n \\ \| \epsilon\|_2 > n^{-3/8}/2}} \exp \brk{- n \frac{D(d,k)}{2} \| \epsilon \|_2^2 (1+o(1))}
  \\
 &= C_n(d,k) \exp \brk{2 n f_1^s\br{\rwnik}} \sum_{\substack{l \in \mathbb{Z}/n \\ l > n^{-3/8}/2}} \sum_{ \substack{\epsilon \in S_n \\ \| \epsilon\|_2 = l}}  \exp \brk{- n l^2 \frac{D(d,k)}{2} (1+o(1))} \\ 
 &= C_n(d,k) \exp \brk{2 n f_1^s\br{\rwnik}} O \br{n^{k^2}} \exp \brk{- \frac{D(d,k)}{2} n^{1/4}} 
 \end{align*}
Consequently, (\ref{eq_sm_S_1}) yields $\Sigma_2=o(\Sigma_1)$.
Thus, we obtain from \Lem~\ref{Lem_hessian_wormald}  that	\begin{align}\label{eq_sm_S_1_1}
\nonumber \cS_1 	&  \sim_{\nu} C_n(d,k) \exp \brk{2 n f_1^s\br{\rwnik}}\sum_{\substack{\rho \in \widehat{\cB}}}   \exp \brk{- n \frac{D(d,k)}{2} \| \epsilon\|_2^2 + o(n^{1/2}) \| \epsilon\|_2}.  \\
&\sim_{\nu} C_n(d,k) \exp \brk{2 n f_1^s\br{\rwnik}}\br{\sqrt{2\pi n}}^{(k-1)^2}k^{-k(k-1)}\br{1-\frac{d}{(k-1)^2}}^{-\frac{(k-1)^2}{2}}. 
	\end{align}
In particular, the last expression  is independent of the choice of the vectors $\rho^0, \rho^1$ that defined $\widehat{\cB}$.
Therefore, substituting~(\ref{eq_sm_S_1_1}) in the decomposition (\ref{eq_sm_Comp}) completes the proof of \Prop~\ref{Prop_second_mom_balanced_exact}.
\end{proof}

 \begin{proof}[Proof of \Prop~\ref{Prop_ratio_sec_first}] %\label{subsec_proof_final}
 First observe that $$ \exp \brk{\sum_{l \geq 2} \lambda_l \delta_l^2}= \left( 1- \frac{d}{(k-1)^2}\right)^{-\frac{(k-1)^2}{2}} \exp \brk{- \frac{d}{2}}.$$
 \Prop~\ref{Prop_ratio_sec_first} is immediately obtained by combining \Lem~\ref{Lem_fm_shift} with \Prop s~\ref{Prop_AchNaor},~\ref{Prop_second_mom_new_v} and \ref{Prop_second_mom_balanced_exact}.
\end{proof}

\subsection{Up to the condensation threshold}

In this last subsection we prove \Prop~\ref{Prop_first_mom_tame_bal}. 
In the regime $2(k-1)\ln(k-1)\leq d^{\prime}<\dc$ for $k\ge k_0$ for some big constant $k_0$, we consider random variables $\Ztame$ instead of $\Zwni$. To prove the proposition we show the following result by adapting our setting in a way that we can apply the second moments argument from~\cite{Danny} and \cite{Cond}.

\begin{proposition}\label{Prop_AcoVil}
Let $\w, \nu \in N$. There is a constant $k_0>3$ such that for
 $k\geq k_0$ and $2(k-1) \ln (k-1) \leq d^{\prime} < \dc$ the following is true.
For each $s \in \Swn$, there exists an integer-valued random variable
	$0 \leq \Ztame \leq \Zwni$ that satisfies $$ \Erw\brk{\Ztame(\gnm)} \sim \Erw \brk{\Zwni(\gnm)}$$ 
and such that for any fixed $\eta >0$ we have $\Erw\brk{\Ztame(\gnm)^2} \leq(1+o(1)) \Erw\brk{\Zwnie(\gnm)}.$
\end{proposition}

In this section we work with the \Erdos-\Renyi~random graph model $G(n,p)$, which is a random graph on $[n]$ vertices where every possible edge is present with probability $p=d/n$ independently. We further assume from now on that $k$ divides $n$. 

The use of results from \cite{Danny, Cond} is complicated by the fact that we are dealing with $(\w,n)$-balanced $k$-colourings that allow a larger discrepancy between the colour classes than \cite{Danny, Cond}, where balanced colourings are defined such that in each color class only a deviation of at most $\sqrt n$ from the typical value $n/k$ is allowed. To circumvent this problem, we introduce the following:
	
Choose a map $\sigma: [n] \to [k]$ uniformly at random and generate a graph $G(n,p',\sigma)$ on $[n]$ by connecting any two vertices $v,w \in [n]$ such that $\sigma(v) \ne \sigma(w)$ with probability $p'=dk/(n(k-1))$ independently.    

Given $\sigma$ and $G(n,p',\sigma)$, we define
\begin{align*}
\alpha_i=|\sigma^{-1}(i)-n/k| \qquad \text{ for } i\in [k]
\end{align*} 
and let $\alpha=\max_{i\in[k]} \alpha_i$. Thus, by definition $\alpha \le \w\sqrt n$. 
We set $n'=n+k\lceil\alpha\rceil$. Further, we let 
\begin{align*}
\beta_i=|\sigma^{-1}(i)-(n+k\lceil\alpha\rceil)/k| \qquad \text{ for } i\in [k].
\end{align*} 
We then construct a coloured graph $G'_{n',p',\sigma'}$ from $G(n,p',\sigma)$ in the following way: 
\begin{itemize}
	\item Add $k\lceil\alpha\rceil$ vertices to $\gnp$ and denote them by $n+1, n+2,...,n+k\lceil\alpha\rceil$.
	\item Define a colouring $\sigma':[n']\to [k]$ by setting $\sigma'(i)=\sigma(i)$ for $i \in [n]$, $\sigma(i)=1$ for $i \in {n+1,...,n+\beta_1}$ and $\sigma(i)=j$ for $j \in \{2,...,k\}$ and $i \in {n+\beta_{j-1}+1,...,n+\beta_j}$.
	\item Add each possible edge $(i,j)$ with $\sigma'(i)\ne\sigma'(j)$ involving a vertex $i \in \{n+1,...,n+k\lceil\alpha\rceil\}$ with probability $p'=dk/(n(k-1))$.
\end{itemize} 
%This construction ensures that $G(n+k\alpha,p',\sigma')$ is equal in distribution to $\gnp$. 
%\vspace{0.2cm}
%Let $G'(n,p')$ be the graph obtained by choosing $k\lceil\alpha\rceil$ vertices uniformly at random from all vertices of $G(n+k\alpha,p',\sigma')$ and deleting them together with all incident edges. Then the follwing is obvious from the construction.
We call a colouring $\tau:[n]\to [k]$ of a graph $G$ on $[n]$ {\em perfectly balanced} if $|\tau^{-1}(i)|=|\tau^{-1}(j)|$ for all $i,j \in [k]$ and we denote the set of all such perfectly balanced colourings by $\widetilde \cB_k(n)$. 
Then the following holds by construction:
\begin{fact}\label{Fact_samedist}
$G'_{n',p',\sigma'}$ has the same distribution as $G(n', p', \tau)$ conditioned on the event that $\tau: [n']\to k$ is perfectly balanced. 
\end{fact}

Let $G''_{n,p',\sigma'{|[n]}}$ denote the graph obtained from $G'_{n',p',\sigma'}$ by deleting the vertices $n+1,...,n+k\lceil\alpha\rceil$ and the incident edges. 

\begin{fact}\label{Fact_samedist_1}
	$G''_{n,p',\sigma'{|n}}$ has the same distribution as $G(n, p',\tau)$ conditioned on the event that $\tau$ is $(\w,n)$-balanced.
\end{fact}

To proceed, we adopt the following notation from \cite{Danny}: Let $\rho\in\cB_k$ be called {\em $s$-stable} if it has precisely $s$ entries bigger than $0.51/k$. Further, let $\bar \cB_k$ be the set of all $\rho\in\cB_k$ such that
$$\sum_{j=1}^k\rho_{ij}=\sum_{j=1}^k\rho_{ji}=1/k\quad\mbox{for all }i\in[k].$$

Then any $\rho\in\bar\cB_k$ is $s$-stable for some $s\in\{0,1,\ldots,k\}$.
In addition, let $\kappa=\ln^{20}k/k$ and let us call $\rho\in\cB_k$ {\em separable} if $k\rho_{ij}\not\in(0.51,1-\kappa)$ for all $i,j\in[k]$.
A $k$-colouring $\sigma$ of a graph $G$ on $[n]$ is called {\em separable}
if for any other $k$-colouring $\tau$ of $G$ the overlap matrix $\rho(\sigma,\tau)$ is separable. We have the following result: 

\begin{lemma}	\label{Lemma_separable}
	Let $s \in \Swn$. There is $k_0>0$ such that for all $k>k_0$ 
	and all $d'$ such that $2(k-1)\ln(k-1)\leq d'\leq(2k-1)\ln k$ the following is true.
	Let $\Ztame(\gnm)$ denote the number of $(\omega,n)$-balanced $k$-colourings of $\gnm$ that fail to be separable.
	Then
	$\Erw[\Ztame(\gnm)]=o(\Erw[\Zwni(\gnm)])$.
\end{lemma}

To prove this lemma, we combine Fact~\ref{Fact_samedist} with \cite[\Lem~3.3]{Danny}. This yields the following\footnote{As a matter of fact, \Lem~3.2 in \cite{Danny} also holds for densities $2(k-1)\ln(k-1)\leq d^{\prime} \le 2(k-1)\ln k-2$, as all steps in the proof are also valid in this regime.}.

\begin{lemma}[\cite{Danny}]	\label{Lemma_sep}
There is $k_0>0$ such that for all $k\ge k_0$ and all $d^{\prime} $ with $2(k-1)\ln(k-1)\leq d^{\prime} \leq(2k-1)\ln k$ each $\tau \in \widetilde \cB_k(n')$ is separable in $G'_{n', p', \tau}$ \whp.
%Let $\Ztame(\gnm)$ denote the number of $k$-colourings $\sigma \in \Bwni$ of $\gnm$ that fail to be separable.
%Then
%	$\Erw\brk{\Ztame(\gnm)}=o\br{\Erw\brk{\Zwni(\gnm)}}$.
\end{lemma}

\begin{proof}[Proof of \Lem~\ref{Lemma_separable}]
Choose a map $\sigma: [n] \to [k]$ uniformly at random and generate a graph $G(n,p',\sigma)$ on $[n]$ by connecting any two vertices $v,w \in [n]$ such that $\sigma(v) \ne \sigma(w)$ with probability $p'$ independently. Construct  $G'_{n',p',\sigma'}$ from $G(n,p',\sigma)$ in the way defined above. Then $\sigma' \in \widetilde\cB_k(n)$.  By \Lem~\ref{Lemma_sep}, $\sigma'$ is separable in $G'_{n',p',\sigma'}$ \whp.
Thus, $\sigma$ is separable in $G''_{n,p',\sigma'{|n}}$ if we define separability using $\kappa'=\frac{\ln^{21} k}{k}$. By choosing $k_0$ large enough and applying Fact~\ref{Fact_samedist_1}, the assertion follows.  	
\end{proof}

For the next ingredient to the proof of \Prop~\ref{Prop_AcoVil}, we need the following definition.
For a graph $G$ on $[n]$ and a $k$-colouring $\sigma$ of $G$, we let
$\cC(G,\sigma)$ be the set of all $\tau\in\cB_k$ that are $k$-colourings of $G$
%{\color{red} such that for each $i\in\brk k$ there is $j\in\brk k$} 
such that $\rho(\sigma,\tau)$ is $k$-stable.

\begin{lemma} \label{Lemma_Clustersize}
	Let $s \in \Swn$. There is $k_0>0$ such that for all $k>k_0$ 
	and all $d^{\prime} $ such that $(2k-1)\ln k-2\leq d^{\prime} \leq\dc$ the following is true.
	There exists an $\eps>0$ such that if $\Ztame(\gnm)$ denotes the number of $(\omega,n)$-balanced $k$-colourings $\sigma$ of $\gnm$ satisfying
	$|\cC(\gnm,\sigma)|>\linebreak \Erw\brk{\Zwni(\gnm)}/\exp\brk{\eps n}$ , then
	$\Erw\brk{\Ztame(\gnm)}=o\br{\Erw\brk{\Zwni(\gnm)}}$.
\end{lemma}

To prove this lemma, we combine~\ref{Fact_samedist} with \cite[\Cor~1.1]{Cond} and obtain the following:

\begin{lemma}[{\cite{Cond}}]\label{Lem_citecond}
	Let $s \in \Swn$. There is $k_0>0$ such that for all $k>k_0$ 
	and all $d^{\prime} $ such that $(2k-1)\ln k-2\leq d^{\prime} \leq\dc$ the following is true.
	Let $\tau \in \widetilde \cB_k(n')$ be a perfectly balanced colour assignment. Then there exists $\eps>0$ such that if $\Ztame(G'_{n',p',\tau})$ denotes the number of $(\omega,n)$-balanced $k$-colourings $\tau$ of $G'_{n',p',\tau}$ satisfying
	$|\cC(G'_{n',p',\tau},\tau)|>\Erw\brk{\Zwni(G'_{n',p',\tau})}/\exp\brk{\eps n}$, then
	$\Erw\brk{\Ztame(G'_{n', p', \tau})}=o\br{\Erw\brk{\Zwni(G'_{n',p',\tau})}}$.
\end{lemma}

\begin{proof}[Proof of \Lem~\ref{Lemma_Clustersize}]
	Choose a map $\sigma: [n] \to [k]$ uniformly at random and generate a graph $G(n,p',\sigma)$ on $[n]$ by connecting any two vertices $v,w \in [n]$ such that $\sigma(v) \ne \sigma(w)$ with probability $p'$ independently. Construct  $G'_{n',p',\sigma'}$ from $G(n,p',\sigma)$ in the way defined above. To construct $G''_{n,p',\sigma'_{|[n]}}$ from $G'_{n',p',\sigma'}$, we have to delete $O(\sqrt n)$ many vertices. By \cite[Section 6]{Cond}, for each of these vertices $v$ we can bound the logarithm of the number of colourings that emerge when deleting $v$ by $O(\ln n)$. Thus, 
	\begin{align}\label{eq_cluster_bigger}
	\ln |\cC(G''_{n,p',\sigma'_{|[n]}},\sigma'_{|[n]})|=\ln |\cC(G'_{n',p',\sigma'},\sigma')|+O(\sqrt n \ln n)=\ln |\cC(G'_{n',p',\sigma'},\sigma')|+o(n).
	\end{align}	
	Then \Lem~\ref{Lemma_Clustersize} follows by combining \Lem~\ref{Lem_citecond} with \eqref{eq_cluster_bigger} and Fact~\ref{Fact_samedist_1}.
\end{proof}

To complete the proof, we have to analyse the function $f_2$ defined in \eqref{eq_def_f_2}, as we know from \eqref{eq_erw_Z_rho} that
\begin{align*}
\Erw\brk{\Zrho^{(2)}(\gnm)}=\exp\brk{nf_2(\rho) +O(\ln n)}.
\end{align*}
The following lemma shows that we can confine ourselves to the investigation  of the function $\bar f_2$ defined in \eqref{eq_f}. 

\begin{lemma}
	Let $\lim_{n\to\infty}(\rho_n)_n=\rho_0$. Then $\lim_{n\to \infty} \ln \Erw\brk{Z^{(2)}_{k,\rho_n}(\gnm)}\le \bar f_2(\rho_0)$. 
\end{lemma}
\begin{proof}
	\Lem~\ref{Lemma_f2bar} yields that 
		\begin{align*}
		\exp\brk{nf_2(\rho)} \sim \exp\brk{n\bar f_2(\rho)+O\br{\w^2}}.
		\end{align*}
		Together with the uniform continuity of $\bar f_2$ this proves the assertion. %Follows because logarithm is differentiable??? 
\end{proof}

We use results from \cite{Danny} where an analysis of $\bar f_2$ was performed. The following lemma summarizes this analysis from~\cite[\Sec~4]{Danny}. The same result was used in \cite{aco_plantsil}.

\begin{lemma}\label{Lemma_Dan}
	For any $c>0$, there is $k_0>0$ such that for all $k>k_0$ 
	and all $d$ such that $(2k-1)\ln k-c\leq d^{\prime} \leq(2k-1)\ln k$ the following statements are true.
	\begin{enumerate}
		\item If $1\leq s<k$, then for all separable $s$-stable $\rho\in\bar\cB_k$ we have $\bar f_2(\rho)<\bar f_2(\bar\rho)$.
		\item If $\rho\in\bar\cB_k$ is $0$-stable and $\rho\neq\bar\rho$, then $\bar f_2(\rho)<\bar f_2(\bar\rho)$.
		\item If $d'=(2k-1)\ln k-2$, then for all separable, $k$-stable $\rho\in\bar\cB_k$ 
		we have $\bar f_2(\rho)<\bar f_2(\bar\rho)$.
	\end{enumerate}
\end{lemma}

\begin{proof}[Proof of~\Prop~\ref{Prop_AcoVil}]
Assume that $k\geq k_0$ for a large enough number $k_0$ and that $d^{\prime} \geq2(k-1)\ln(k-1)$.
%An welcher Stelle müssen wir den Übergang zu doppelt stochastischen overlap matrizen machen? Brauchen wir den überhaupt? Langt es, zu sagen, dass $f$ uniform continuous ist auf $\cB_k$ (in Fall 1 getan)?
We consider two different cases.
\begin{description}
\item[Case 1: $d^{\prime} \leq(2k-1)\ln k \kern-.2ex -\kern-.2ex 2$]
 Let $\Ztame$ be the number of $(\w,n)$-balanced separable $k$-colourings of $\gnm$.
 Then \Lem~\ref{Lemma_sep} implies that 
$ \Erw[ \Ztame(\gnm) ] \sim \Erw \brk{\Zwni(\gnm)}$.
Furthermore, in the case that $d^{\prime} =(2k-1)\ln k-2$, the combination of the statements of \Lem~\ref{Lemma_Dan} imply that
$\bar f_2(\rho)<\bar f_2(\bar\rho)$ for any separable $\rho\in\bar \cB_k\setminus\cbc{\bar\rho}$.
As $\bar f_2(\rho)$ is the sum of the concave function $\rho\mapsto \cH(\rho)$ and the convex function $\rho\mapsto\frac d2\ln(1-2/k\norm\rho_2^2)$,
this implies that, in fact,
for any $d'\leq(2k-1)\ln k-2$ we have  $\bar f_2(\rho)<\bar f_2(\bar\rho)$ for any separable $\rho\in\bar \cB_k\setminus\cbc{\bar\rho}$.
	Hence, the uniform continuity of $\bar f_2$ on $\cB_k$ 
	and \eqref{eq_erw_Z_rho} yield
		\begin{align}\label{eq_Prop_aco_vil_1}
		\Erw[\Zwni(\gnm)^2]\leq(1+o(1))
			\sum_{ \substack{\rho \in \Bwni\\\rho\mbox{\ \scriptsize is $0$-stable}}}
				\Erw \brk{\Zrho^{(2)}(\gnm)}.
		\end{align}
	Additionally, as $\bar \cB_k$ is a compact set, with the second statement of \Lem~\ref{Lemma_Dan} it follows that for any $\eta>0$ there exists $\eps>0$ such that 
	\begin{align}\label{eq_Dan_1}
	\max_{ \substack{\rho \in \Bwni\\\rho\mbox{\ \scriptsize is $0$-stable}\\\norm{\rho-\bar\rho}_2>\eta}}\exp\brk{n\bar f_2(\rho)} \le \exp\brk{n(\bar f_2(\bar \rho)-\eps)}.
	\end{align} 
	As on the other hand it holds that
	\begin{align}\label{eq_Dan_2}
	\Erw\brk{\Zwnie(\gnm)} \ge \exp\brk{n\bar f_2(\bar \rho)}/\mathrm{poly}(n),
	\end{align}
	combining~\eqref{eq_Dan_1} and \eqref{eq_Dan_2} with \eqref{eq_erw_Z_rho} and the observation that $|\Bwni|\le n^{k^2}$, we see that for any $\eta>0$,
			\begin{align}\label{eq_Prop_aco_vil_2}
			\sum_{ \substack{\rho \in \Bwni\\\rho\mbox{\ \scriptsize is $0$-stable}\\\norm{\rho-\bar\rho}_2>\eta}}
				\Erw \brk{\Zrho^{(2)}(\gnm)}&\leq
					\sum_{ \substack{\rho \in \Bwni\\\rho\mbox{\ \scriptsize is $0$-stable}\\\norm{\rho-\bar\rho}_2>\eta}}
				\exp\brk{n\bar f_2(\rho)+O(\ln n)}=o\br{\Erw\brk{\Zwnie(\gnm)}}.
		\end{align}
\item[Case 2: $(2k-1)\ln k-2<d^{\prime} <\dc$]
	 For an appropriate $\eps>0$ let $\Ztame$ be the number of $(\w,n)$-balanced separable $k$-colourings $\sigma$ of $\gnm$
	 such that $|\cC(\gnm,\sigma)|\leq\Erw\brk{\Zwni(\gnm)}/\exp\brk{\eps n}$.
	 Then \Lem s~\ref{Lemma_sep} and~\ref{Lemma_Clustersize} imply that
	 		$ \Erw[ \Ztame(\gnm) ] \sim \Erw \brk{ \Zwni(\gnm)}$.
	Furthermore, the 
	first part of \Lem~\ref{Lemma_Dan} and equation \eqref{eq_erw_Z_rho} entail that (\ref{eq_Prop_aco_vil_1}) holds for this random variable $\Ztame$.
	Moreover, as in the previous case
	(\ref{eq_Dan_1}), (\ref{eq_Dan_2}), \eqref{eq_erw_Z_rho} and the second part of \Lem~\ref{Lemma_Dan} show that~(\ref{eq_Prop_aco_vil_2}) holds true for any fixed $\eta>0$.
\end{description}
In either case the assertion follows by combining~(\ref{eq_Prop_aco_vil_1}) and~(\ref{eq_Prop_aco_vil_2}).
\end{proof}
 \begin{proof}[Proof of \Prop~\ref{Prop_first_mom_tame_bal}]
 	The assertion is obtained by combining \Prop~\ref{Prop_first_mom} with \Prop s~\ref{Prop_AcoVil},~\ref{Prop_second_mom_new_v} and \ref{Prop_second_mom_balanced_exact}.
 \end{proof}

\subsubsection*{Acknowledgements}
I thank my PhD advisor Amin Coja-Oghlan for constant support and valuable suggestions and Samuel Hetterich for helpful discussions.

\appendix
\section{}\label{appendix}

\begin{proof}[Proof of \Cor~\ref{Cor_first_moment_tame_bal}] This proof is a close adaption of the analogous proof in \cite{aco_plantsil}. We let $\cE$ denote the event $\cbc{\forall 2 \leq l\leq L:C_{l,n}=c_l}$ and fix $s \in \Swn$. Let $\cZ_n=\Ztame(\gnm)$ for the sake of brevity.
	Since $\cZ_n\leq\Zwni(\gnm)$, equation (\ref{eq_first_moment_tame_bal}) yields the upper bound
	\begin{align}\label{eq_cor_first_moment}
	\frac{\Erw\brk{\cZ_n|\cE}}{\Erw\brk{\cZ_n}}
	&\leq \frac{\Erw\brk{\Zwni(\gnm)|\cE}}{(1+o(1))\Erw\brk{\Zwni(\gnm)}}
	\sim\prod_{l=2}^L\brk{1+\delta_l}^{c_l}\exp\brk{-\delta_l\lambda_l}.
	\end{align}
	We show the following matching lower bound:
	\begin{align}\label{eq_Cor_count}
	\Erw\brk{\cZ_n|\cE}\geq(1-o(1))\Erw\brk{\Zwni(\gnm)|\cE}. 
	\end{align}
	Indeed, assume for contradiction that~(\ref{eq_Cor_count}) is false. Then we can find an $n$-independent $\eps>0$ such that for infinitely many~$n$,
	\begin{align}\label{eq_Cor_count_1}
	\Erw\brk{\cZ_n|\cE}<(1-\eps)\Erw\brk{\Zwni(\gnm)|\cE}. 
	\end{align}
	By Fact~\ref{Fact_cyc} there exists an $n$-independent $\xi=\xi(c_2,\ldots,c_L)>0$ such that $\pr\brk{\cE}\geq\xi$.
	Hence, (\ref{eq_Cor_count_1}) and Bayes' formula imply that
	\begin{align}\label{eq_Cor_count_2}
	\nonumber \Erw\brk{\cZ_n}&=\Erw\brk{\cZ_n|\cE}\pr\brk \cE+\Erw\brk{\cZ_n|\neg \cE}\pr\brk{\neg \cE}\\
	&\leq(1-\eps)\Erw\brk{\Zwni(\gnm)|\cE}\pr\brk \cE+\Erw\brk{\Zwni(\gnm)|\neg \cE}\pr\brk{\neg \cE}\nonumber\\
	&\leq\Erw[\Zwni(\gnm)]-\eps\xi\cdot\Erw[\Zwni(\gnm)|\cE]\nonumber\\
	&=\Erw[\Zwni(\gnm)]\cdot\bc{1+o(1)-\eps\xi\prod_{l=2}^L[1+\delta_l]^{c_l}\exp\brk{-\delta_l\lambda_l}}\nonumber\\
	&=(1-\Omega(1))\Erw\brk{\Zwni(\gnm)},
	\end{align}
	where the last equality holds since $\delta_l,\lambda_l$ and $c_l$ remain fixed as $n\to\infty$. As~(\ref{eq_Cor_count_2}) contradicts~(\ref{eq_first_moment_tame_bal}), we have established~(\ref{eq_Cor_count}).
	Finally, combining~(\ref{eq_Cor_count}) with~(\ref{eq_CondRatio1stMom}) and~(\ref{eq_first_moment_tame_bal}), we get
	\begin{align}\label{eq_Cor_count_3}
	\frac{\Erw\brk{\cZ_n|S}}{\Erw\brk{\cZ_n}}
	&\geq \frac{(1-o(1))\Erw\brk{\Zwni(\gnm)|S}}{(1+o(1))\Erw\brk{\Zwni(\gnm)}}\sim
	\prod_{l=2}^L\brk{1+\delta_l}^{c_l}\exp\brk{-\delta_l\lambda_l},
	\end{align}
	and the assertion follows from~(\ref{eq_cor_first_moment}) and~(\ref{eq_Cor_count_3}).
\end{proof}

\subsection{Calculating the first moment}\label{Sec_appendix_first}
The following proofs are very close to analogous proofs in \cite{aco_plantsil}.	

\begin{proof}[Proof of \Lem~\ref{Lem_first_mom_balanced}]

	As the edges in $\gnm$ are independent by construction, the expected number of $k$-colourings with colour density $\rho$ is given by
	\begin{equation}\label{eq_proof_first_moment_balanced_1}
	\Erw\brk{\Zrho(\gnm)}= {n \choose \rho_1 n , \dots, \rho_k n }\bc{1-\frac1N\sum_{i=1}^k\bink{\rho_in}{2}}^m,\quad\mbox{where }N=\bink n2.
	\end{equation}
	Further, the number of forbidden edges is given by
	\begin{equation*} 
	\sum_{i=1}^k { \rho_{i } n \choose 2} = N \br{\sum_{i=1}^k \rho_{i }^2}+ \frac{n}{2} \br{\sum_{i=1}^k \rho_{i }^2-1}+ O(1)
	\end{equation*}
	and thus
	\begin{align}\label{eq_proof_first_moment_balanced_2}
	\nonumber m \ln \br{1- \frac{1}{N} \sum_{i=1}^k { \rho_{i } n \choose 2}}& = m \ln \brk{\br{1 + \frac{n}{2N}} \br{1- \sum_{i=1}^k \rho_{i }^2}}+ o(1)\\ 
	& = n \frac{d}{2} \ln\br{1-\sum_{i=1}^k \rho_i^2} + \frac{d}{2} + o(1).
	\end{align}
	
	Equation \eqref{eq_first_mom_balanced_1} follows from (\ref{eq_proof_first_moment_balanced_1}), (\ref{eq_proof_first_moment_balanced_2}) and Stirling's formula.
	Moreover, (\ref{eq_first_mom_balanced_2}) follows from~(\ref{eq_proof_first_moment_balanced_1}) 
	and~(\ref{eq_proof_first_moment_balanced_2}) because $\norm{\rho-\rho^\star}_2=o(1)$ implies that $\sum_{i=1}^k\rho_i^2\sim1/k$ and
	\begin{equation*} 
	{ n \choose \rho_1 n , \dots, \rho_k n } \sim \br{{2 \pi n}}^{\frac{1-k}{2}} k^{k/2} \exp\brk{n \cH(\rho)}.
	\end{equation*}
\end{proof}

\begin{proof}[Proof of \Cor~\ref{Cor_first_mom_total}]
	The functions  $\rho\in\cA_k \mapsto \cH(\rho)$ and $\rho\in\cA_k \mapsto \frac{d}{2} \ln ( 1 - \sum_{i=1}^k \rho_i^2 )$
	are both concave and attain their maximum at $\rho = \rho^\star$.
	Consequently, setting
	$B(d,k) = k(1 + \frac{d}{k-1})$ and expanding around $\rho = \rho^\star$, we obtain
	\begin{align}\label{eq_proof_fm_tot}
	f_1\br{\rho^\star} - \frac{B(d,k)}{2} \| \rho - \rho^\star \|_2^2 - O\br{\| \rho - \rho^\star \|_2^3} \leq  f_1(\rho) \leq f_1\br{\rho^\star}  - \frac{B(d,k)}{2} \| \rho - \rho^\star \|_2^2. 
	\end{align}
	Plugging the upper bound from~(\ref{eq_proof_fm_tot}) into~(\ref{eq_first_mom_balanced_1}) and observing that $|\cA_k(n)|\leq n^{k}=\exp\brk{o(n)}$, we find
	
	\begin{align}\label{eq_proof_fm_tot_2}
	S_1=\sum_{\substack{ \rho \in \cA_k(n) \\\| \rho - \rho^\star \|_2 > n^{-3/8}} } 
	\Erw \brk{\Zrho(\gnm)} \leq C_2   \exp \brk{ f_1\br{\rho^\star}} \exp \brk{- \frac{B(d,k)}{2} n^{1/6}}. 
	\end{align}
	On the other hand, (\ref{eq_first_mom_balanced_2}) implies that
	\begin{align} \label{eq_proof_fm_tot_3} 
	\nonumber 	S_2&=\sum_{\substack{ \rho \in \cA_k(n) \\\| \rho - \rho^\star \|_2 \leq n^{-3/8}} } \Erw \brk{\Zrho(\gnm)}
	\sim  \sum_{\substack{ \rho \in \cA_k(n) \\
			\| \rho - \rho^\star \|_2 \leq n^{-3/8}} } \br{{2 \pi n}}^{\frac{1-k}{2}} k^{k/2} \exp\brk{d/2}  \exp \brk{n f_1(\rho)}\\ 
	& \sim \br{{2 \pi n}}^{\frac{1-k}{2}} k^{k/2} \exp\brk{d/2+ n f_1\br{\rho^\star}}		 	\sum_{\substack{ \rho \in\cA_k(n)}} \exp \brk{-n \frac{B(d,k)}{2} \| \rho - \rho^\star \|_2^2}. 
	\end{align}
	The last sum is nearly in the standard form of a Gaussian summation, merely the vectors $\rho\in\cA_k(n)$ that we sum over are subject to the linear constraint $\rho_1+\cdots+\rho_k=1$.
	We rid ourselves of this constraint by substituting $\rho_k=1-\rho_1-\cdots-\rho_{k-1}$.
	Formally, let $J$ be the $(k-1)\times(k-1)$-matrix whit diagonal entries equal to $2$ and remaining entries equal to $1$. We observe that $\det J=k$. Then 
	\begin{align}\label{eq_proof_fm_tot_4}
	\nonumber \sum_{\substack{\rho \in \cA_k(n)}} \exp\brk{-n \frac{B(d,k)}{2} \| \rho - \rho^\star \|_2^2}   
	&\sim 	\sum_{y \in \frac1n\mathbb{Z}^k} \exp \brk{-n \frac{B(d,k)}{2}\scal{Jy}y}\\
	&\sim\br{2 \pi n}^{\frac{k-1}{2}} k^{-\frac{k}{2}}\br{1 + \frac{d}{k-1}}^{- \frac{k-1}{2}}.
	\end{align}
	Plugging~\eqref{eq_proof_fm_tot_4} into \eqref{eq_proof_fm_tot_3}, we obtain
	\begin{align}\label{eq_proof_fm_tot_5} 
	\nonumber    S_2 &\sim \br{{2 \pi n}}^{\frac{1-k}{2}} k^{k/2} \exp \brk{d/2+ n f_1(\rho^\star)} \br{2 \pi n}^{\frac{k-1}{2}} k^{- \frac{k}{2}} \br{1 + \frac{d}{k-1}}^{- \frac{k-1}{2}}\\
	&=\exp \brk{d/2+ n f_1(\rho^\star)}\br{1 + \frac{d}{k-1}}^{- \frac{k-1}{2}}.
	\end{align}
	Finally, comparing~(\ref{eq_proof_fm_tot_2}) and~(\ref{eq_proof_fm_tot_5}), 
	we see that $S_1=o(S_2)$.
	Thus, $\Erw [ Z_k(\gnm)] =S_1+S_2\sim S_2$, and the assertion follows from~(\ref{eq_proof_fm_tot_5}).
\end{proof}

\subsection{Calculating the second moment}\label{Sec_appendix_second}
The following proofs are very close to analogous proofs in \cite{aco_plantsil}.	

\begin{proof}[Proof of \ref{Fact_varZ}]
	To calculate the expected number of pairs of colourings $\sigma, \tau$ with overlap $\rho \in \cB_k(n)$, we first observe that
	\begin{align*}
	\pr\brk{\sigma, \tau \textrm{ are $k$-colourings of $\gnm$}}=\br{1-\frac{\cF(\sigma, \tau)}{N}}^m,
	\end{align*}
	where $\cF(\sigma, \tau)$ is the number of ``forbidden'' edges joining two vertices with the same colour under either $\sigma$ or $\tau$ and  $N=\binom{n}{2}$. We have
	\begin{align*}
	\cF(\sigma, \tau) &= \sum_{i=1}^k { \rho_{i \star} n \choose 2}+ \sum_{j=1}^k { \rho_{\star j} n \choose 2} - \sum_{i,j=1}^k { \rho_{ij} n \choose 2} 
	\\ & = N \br{\sum_{i=1}^k \rho_{i \star}^2 + \sum_{j=1}^k \rho_{\star j}^2 - \sum_{i,j = 1}^k \rho_{ij}^2} + \frac{n}{2} \br{\sum_{i=1}^k \rho_{i \star}^2+\sum_{j=1}^k \rho_{\star j}^2 - \sum_{i,j=1}^k \rho_{ij}^2-1}+ O(1)
	\end{align*}
	and thus, the probability that $\sigma$ and $\tau$ are both colourings of $\gnm$ only depends on their overlap $\rho$ and is given by
	\begin{align}\label{eq_Z1} 
	\pr\brk{\sigma, \tau \textrm{ are $k$-colourings of $\gnm$}}
	& \sim  \exp\brk{m \ln \br{1 - \sum_{i=1}^k  \rho_{i \star}^2- \sum_{j=1}^k  \rho_{\star j}^2  + \sum_{i,j=1}^k \rho_{ij}^2} + \frac{d}{2}}.
	\end{align}
	
	It remains to multiply this by the total number of $\sigma, \tau$ with overlap $\rho \in \cB_k(n)$. By Stirling's formula, this number is given by
	\begin{align} \label{eq_Z2}
	{ n \choose \rho_{11} n, \dots, \rho_{kk} n} \sim \sqrt{2 \pi} n^{- \frac{k^2-1}{2}} \br{\prod_{i,j} \frac{1}{\sqrt{2 \pi \rho_{ij}}}} \exp\brk{n  \cH(\rho)}.
	\end{align}
	
	Equation (\ref{eq_Zrho}) is obtained by combining (\ref{eq_Z1}) and (\ref{eq_Z2}). To prove (\ref{eq_Z_om}), we observe that if $\|\rho-\bar \rho\|_2^2=o(1)$, we have
	\begin{align*}
	\frac{\sqrt{2 \pi} n^{\frac{1-k^2}{2}}} {\prod_{i,j=1}^k \sqrt{2 \pi \rho_{ij}}} \sim k^{k^2} \br{2 \pi n}^{\frac{1-k^2}{2}}
	\end{align*}
	and the statement follows.
\end{proof}	

\begin{proof}[Proof of \Lem~\ref{Lem_hessian_wormald}] 
	Together with the Euler-Maclaurin formula and \Lem~\ref{Lemma_WormaldMatrix}, a Gaussian integration yields
	\begin{align*}
	\sum_{\epsilon \in S_n} \exp &\brk{- n \frac{D}{2} \| \epsilon\|_2^2 + o(n^{1/2}) \| \epsilon\|_2}=
	\sum_{\epsilon \in \left( \mathbb{Z}/n \right)^{(k-1)^2} }
	\exp \brk{- n \frac{D}{2}\sum_{i,j,i',j'\in[k-1]}\cH_{(i,j),(i',j')}\eps_{ij}\eps_{i'j'} + o(n^{1/2}) \| \epsilon \|_2} \\ 
	& \sim n^{(k-1)^2} \int \dots  \int 
	\exp \brk{- n \frac{D}{2}\sum_{i,j,i',j'\in[k-1]}\cH_{(i,j),(i',j')}\eps_{ij}\eps_{i'j'}} 
	\mathrm d\eps_{11}\cdots\mathrm d\eps_{(k-1)(k-1)}
	\\ & \sim \left( \sqrt{2 \pi n} \right)^{(k-1)^2} D^{\frac{-(k-1)^2}{2}} (\det\cH)^{-1/2} 
	\sim \left( \sqrt{2 \pi n} \right)^{(k-1)^2} D^{\frac{-(k-1)^2}{2}} k^{-(k-1)} %
	, \end{align*} 
	as desired.
\end{proof}		

\subsection{Counting short cycles}\label{Sec_Short_cyc}

\noindent
In this section we count the number of cycles of a short fixed length in order to prove \Prop~\ref{Prop_CondRatio1stMom}. The results in this section were already obtained in \cite{aco_plantsil} and the proofs are a very close adaption of the ones in \cite{aco_plantsil}. We recall that for $l=2,\ldots,L$ we denoted by $C_{l,n}$ the number of cycles of length exactly $l$ in $\gnm$. We let $c_2,\ldots,c_L$ be a sequence of non-negative integers and $\cE$ the event that $C_{l,n}=c_l$ for $l=2,\ldots,L$. We recall $\lambda_l,\delta_l$ from~(\ref{eq_lambdelt}). For a map $\sigma:[n]\mapsto [k]$, we define $\cV(\sigma)$ as the event that $\sigma$ is a $k$-colouring of the random graph $\gnm$. Our starting point is the following lemma concerning the distribution of the random variables $C_{l,n}$ given $\cV(\sigma)$.

\begin{lemma}\label{Lem_planted_cyc}
	Let $\mu_l=\frac{d^l}{2l}\brk{1+\frac{(-1)^l}{(k-1)^{l-1}}}$.
	Then
	$\pr[\cE|\cV(\sigma)] \sim \prod_{l=2}^{L}\frac{\exp\brk{-\mu_l}}{c_l!}\mu_l^{c_l}$ for any $\sigma$ with $\rho(\sigma)\in\Aw$.
\end{lemma}
\noindent

\begin{proof}
	All we have to show is that for any fixed sequence of integers $m_2, \ldots, m_L\geq0$, the joint factorial moments satisfy 
	\begin{align}\label{eq_JointFactMom}
	\Erw\brk{(C_{2,n})_{m_2}\cdots (C_{L,n})_{m_L}|\cV(\sigma)}\sim
	\prod_{l=2}^L \mu_l^{m_l}. 
	\end{align}
	Then \Lem~\ref{Lem_planted_cyc} follows from~\cite[\Thm~1.23]{Bollobas}.
	
	To establish \eqref{eq_JointFactMom}, we interpret $(C_{2,n})_{m_2}\cdots (C_{L,n})_{m_L}$ as the number of sequences of $m_2+\cdots+m_L$ distinct cycles such that 
	$m_2$ is the number of cycles of length $2$, and so on. 
	We let $Y$ be the number of those sequences of cycles such that any two cycles are vertex-disjoint and $Y'$ be the number of sequences having intersecting cycles. Obviously, we have 
	\begin{align}\label{eq_SplitToDisjointJoint}
	\Erw\brk{(C_{2,n})_{m_2}\cdots (C_{L,n})_{m_L}| \cV(\sigma)}=\Erw\brk{Y|\cV(\sigma)}+\Erw\brk{Y'|\cV(\sigma)}.
	\end{align}
	For $\Erw\brk{Y'|\cV(\sigma)}$, we use the following claim that we prove at the end of this section.
	
	\begin{claim}\label{Claim_OverlapCyc}
		It holds that $\Erw\brk{Y'|\cV(\sigma)}=O(n^{-1})$.
	\end{claim}
	\noindent
	
	Thus, it remains to count the number of vertex disjoint cycles conditioned on $\cV(\sigma)$. The line of arguments we use is similar to \cite[\Sec~2]{Kemkes}. To simplify the calculations, we define $D_{l,n}$ as the number of rooted, directed cycles of length $l$ in $\gnm$, implying that $D_{l,n}=2lC_{l,n}$.
	
	For a rooted directed cycle $(v_1,\ldots,v_l)$ of length $l$, we call $(\sigma(v_1),\ldots,\sigma(v_l))$ the {\em type} of the cycle under $\sigma$. Let $D^t_{l,n}$ denote the number of rooted, directed cycles of length $l$ and type $t=(t_1,...,t_l)$.  
	We claim that
	\begin{align}\label{eq_disjoint_cyc}
	\Erw\brk{D^t_{l,n} |\cV(\sigma)}
	\sim\bcfr nk^l\frac{(m)_l}{(N-\Forb(\sigma))^l}\sim\bcfr{d}{k-1}^l\quad\mbox{with }N=\bink n2.
	\end{align}
	Indeed, as $\sigma$ is $(\w,n)$-balanced, the number of ways of choosing $l$ vertices $(v_1,\ldots,v_l)$ such that $\sigma(v_i)=t_i$ for all $i$ is $(1+o(1))(n/k)^l$ and each edge $\cbc{v_i,v_{i+1}}$ of the cycle is present in the graph with a probability asymptotically equal to $m / (N-\Forb(\sigma))$. This explains the first asymptotic equality in~(\ref{eq_disjoint_cyc}).
	The second one follows because $m=dn/2$ and $\Forb(\sigma)\sim N/k$.
	
	In particular, the r.h.s.\ of~(\ref{eq_disjoint_cyc}) is independent of the type $t$.
	For a given $l$, let $T_l$ signify the number of all possible types of cycles of length $l$.
	Thus, $T_l$ is the set of all sequences $(t_1,\ldots,t_l)$ such that $t_{i+1}\neq t_i$ for all $1\leq i<l$ and $t_l\neq t_1$.
	Let $T_1=0$.
	Then $T_l$ satisfies the recurrence 
	\begin{align}\label{eq_rec}
	T_l+T_{l-1}=k(k-1)^{l-1}.
	\end{align}
	To see this, observe that $k(k-1)^{l-1}$ is the number of all sequences
	$(t_1,\ldots,t_l)$ such that $t_{i+1}\neq t_i$ for all $1\leq i<l$.
	Any such sequence either satisfies $t_l\neq t_1$, which is accounted for by $T_l$, or $t_l=t_1$ and $t_{l-1}\neq t_1$,
	in which case it is contained in $T_{l-1}$. 
	
	Hence, iterating \eqref{eq_rec} gives $T_l=(k-1)^l+(-1)^l(k-1)$.
	Combining this formula with~\eqref{eq_disjoint_cyc}, we obtain
	\begin{align*} 
	\Erw\brk{D_{l,n} |\cV(\sigma)}&\sim T_l \cdot \Erw\br{D^t_{l,n} |\cV(\sigma)} 
	\sim d^{l}\br{1+\frac{(-1)^l}{(k-1)^{l-1}}}.
	\end{align*}
	Recalling that $C_{l,n}=D_{l,n}/(2l)$, we get
	\begin{align}\label{eq_disjoint_cyc_2}
	\Erw\brk{C_{l,n} |\cV(\sigma)}&\sim \frac{d^l}{2l}\br{1+\frac{(-1)^l}{(k-1)^{l-1}}}.
	\end{align}
	Since $Y$ considers only vertex disjoint cycles and $l$, $m_2,\ldots, m_L$ remain fixed as $n\ra\infty$, equation \eqref{eq_disjoint_cyc_2} yields
	\begin{align*}
	\Erw\brk{Y|\cV(\sigma)} \sim\prod_{l=2}^L \br{\frac{d^l}{2l}\br{1+\frac{(-1)^l}{(k-1)^{l-1}}}}^{m_l}.
	\end{align*}
	Plugging the above relation and Claim \ref{Claim_OverlapCyc} into (\ref{eq_SplitToDisjointJoint}), we get 
	(\ref{eq_JointFactMom}) and the assertion follows. 
\end{proof}

\begin{propositionproof}{ \ref{Prop_CondRatio1stMom}} 
	Let $s \in \Swn$. By Bayes' rule and \Lem~\ref{Lem_planted_cyc} we have
	\begin{align*}
	\Erw\brk{\Zwni(\gnm)| \cE}&=\frac{1}{\pr[\cE]}\sum_{\tau\in\Awni}\pr[\cV(\tau)] \pr\brk{\cE|\cV(\tau)} \\
	&\sim \frac{
		\prod_{l=2}^{L}\frac{\exp\brk{-\mu_l}}{c_l!}\mu_l^{c_l}}{\pr[\cE]}
	\sum_{\tau\in \Awni}\pr\brk{\cV(\tau)}\\
	&\sim \frac{\prod_{l=2}^{L}\frac{\exp\brk{-\mu_l}}{c_l!}\mu_l^{c_l}}{\pr[\cE]}\Erw\brk{\Zwni(\gnm)}. 
	\end{align*}
	From \Lem~\ref{Lem_planted_cyc} and Fact~\ref{Fact_cyc} we get that
	\begin{align*}
	\frac{\prod_{l=2}^{L}\frac{\exp\brk{-\mu_l}}{c_l!}\mu_l^{c_l}}{\pr[\cE]}\sim\prod_{l=2}^L\brk{1+\delta_l}^{c_l}\exp\brk{-\delta_l\lambda_l}, 
	\end{align*}
	whence \Prop~\ref{Prop_CondRatio1stMom} follows. 
\end{propositionproof}

\begin{claimproof}{\ref{Claim_OverlapCyc}}
	For every subset $R$ of $l\le L$ vertices, let  $\mathbb{I}_{R}$ be equal to 1 if the number of edges with both ends in $R$ is at least $|R|+1$. 
	Let $H_L$ be the event that $\{\sum_{R:|R|\leq L}\mathbb{I}_{R}>0\}$. By definition, if $Y'>0$ then the event $H_L$ occurs. This implies that
	\begin{align*}
	\pr\brk{Y'>0|\cV(\sigma)}\leq \pr\brk{H_L|\cV(\sigma)}.
	\end{align*}
	Thus, it suffices to appropriately bound $\pr[H_L|\cV(\sigma)]$. Markov's inequality yields 
	\begin{align*}
	\pr\brk{H_L|\cV(\sigma)}&\leq \Erw\brk{\sum_{R:|R|\leq L}\mathbb{I}_{R}|\cV(\sigma)}=\sum_{l=2}^L\sum_{R:|R|=l}\Erw\brk{\mathbb{I}_{R}|\cV(\sigma)}.
	\end{align*}
	For any set $R$ such that $|R|=l$, 
	we can put $l+1$ edges inside  the set in at most ${{l\choose 2}\choose l+1}$ ways.
	Clearly conditioning on $\cV(\sigma)$ can only reduce the  number of different placings of the edges.
	For a fixed set $R$ of cardinality $l$, we get, using inclusion/exclusion and the Binomial theorem as well as the fact that $\cF\bc{\sigma}\sim N/k$:
	\begin{align*}
	\Erw\brk{\mathbb{I}_{R}|\cV(\sigma)}&\leq {{l\choose 2}\choose l+1 }    
	{\sum_{i=0}^{l+1}{l+1 \choose i}(-1)^i\left(1-\frac{i}{N-\cF\bc{\sigma}}\right)^m}\\
	&\leq {{l\choose 2}\choose l+1} \br{\frac{m}{N-\cF\bc{\sigma}}}^{l+1}
	\sim {{l\choose 2}\choose l+1} \br{\frac{d}{n(1-1/k)}}^{l+1}.
	\end{align*}
	As ${i\choose j}\leq \left(ie/j\right)^j$, it follows that
	\begin{align*}
	\pr\brk{H_L|\cV(\sigma)} &\leq (1+o(1))\sum_{l=2}^L{n \choose l}{{l\choose 2}\choose l+1} \br{\frac{d}{n(1-1/k)}}^{l+1}\\
	&\leq(1+o(1)) \sum_{l=2}^L\br{\frac{ne}{l}}^l\br{\frac{le}{2}}^{l+1} \br{\frac{d}{n(1-1/k)}}^{l+1}\\
	&\leq \frac{1+o(1)}{n} \sum_{l=2}^L\frac{led}{2(1-1/k)}\br{\frac{e^2d}{2(1-1/k)}}^l  =O(n^{-1}),
	\end{align*}
	where the last equality holds since $L$ is a fixed number. The proves the claim.
\end{claimproof}

\end{document}